\documentclass[10pt]{article}

\usepackage{scrextend}
\newif\ifprintver
\printvertrue

\ifprintver
\usepackage{pgfpages}
\fi              
\usepackage[parfill]{parskip}    
\usepackage{graphicx}
\usepackage{amssymb}

\usepackage[utf8]{inputenc}
\usepackage{minitoc}
\usepackage{graphicx}
\usepackage{multicol}
\usepackage{amssymb}
\usepackage{amsmath}
\usepackage{amsbsy}
\usepackage{amsthm}
\usepackage{fancyhdr}
\usepackage{appendix}
\usepackage{hyperref}
\usepackage[all]{xy}

\theoremstyle{plain}
\newtheorem{ej}{Example}
\theoremstyle{definition}
\newtheorem{defi}{Definition}

\newtheorem{teo}{Theorem}
\newtheorem{prop}{Proposition}
\newtheorem{cor}{Corollary}
\newtheorem{obs}{Remark}

\newtheorem{lem}{Lemma}

\def\Stirling2#1#2{\left\{{\begin{matrix}#1\\#2\end{matrix}}\right\}}
\def\Yoneda{\mathbf{Y}}
\def\yoneda{\mathbf{y}}

\def\Set{\mathbf{Sets}}

\title{Augmented simplicial combinatorics through category theory: subdivisions and cyclinders}
\title{Semi-simplicial combinatorics of cyclinders and subdivisions}

\author{Jos\'e Manuel Garc\'{\i}a Calcines, \\ Luis Javier Hern\'andez Paricio and\\  Mar\'{\i}a Teresa Rivas Rodr\'{\i}guez}

\date{}

\begin{document}

\maketitle

\begin{abstract}

In this work, we analyze the combinatorial properties of cylinders and subdivisions of augmented semi-simplicial sets. These constructions are obtained as particular cases of a certain action from a co-semi-simplicial set on an augmented semi-simplicial set.
We also consider cylinders and subdivision operators in the algebraic setting of augmented sequences of integers. These operators are defined either by taking an action of matrices on sequences of integers (using binomial matrices) or by taking the simple product of sequences and matrices.
We compare both the geometric and algebraic contexts using the sequential cardinal functor $|\cdot|$, which associates the augmented sequence $|X|=(|X_n|)_{n\geq -1}$ to each augmented semi-simplicial finite set $X$. Here, $|X_n|$ stands for the finite cardinality of the set of $n$-simplices $X_n$.
The sequential cardinal functor transforms the action of any co-semi-simplicial set into the action of a matrix on a sequence. Therefore, we can easily calculate the number of simplices of cylinders or subdivisions of an augmented semi-simplicial set. Alternatively, instead of using the action of a matrix on a sequence, we can also compute suitable matrices and consider the product of an augmented sequence of integers and an infinite augmented matrix of integers.
The calculation of these matrices is related mainly to binomial, chain-power-set, and Stirling numbers. From another point of view, these matrices can be considered as continuous automorphisms of the Baer-Specker topological group.

\end{abstract}



\tableofcontents


\section{Introduction}

In our previous work \cite{GHR2022}, we analyzed some combinatorial properties of the category of augmented semi-simplicial sets. We studied the sequential cardinal functor, which associates an augmented sequence $|X|_n=|X_n|$ of non-negative integers with each augmented semi-simplicial finite set $X$. The categories of augmented semi-simplicial sets and augmented sequences of integers admit monoidal structures induced by adequate join products and unit objects.
This fact enabled us to easily calculate the number of simplices of cones and suspensions of an augmented semi-simplicial set, as well as other augmented semi-simplicial sets that are built by joins.

The more standard category of simplicial sets is also used as a model to study homotopy invariants. The tools used to model sets of homotopy classes in the category of simplicial sets are based on cylinders and subdivisions (the subdivision approximation theorem). For this reason, the main objective of this paper is to analyze the combinatorial properties of subdivisions and cylinders in the category of augmented semi-simplicial sets.

Among others, we can highlight the following goals:

\begin{itemize}
\item The study of some functorial procedures for constructing cylinders and subdivisions for augmented semi-simplicial sets.
\item The analysis of the numerical sequences (sequential cardinals) arising from the combinatorial structure of cylinders and subdivisions of augmented semi-simplicial sets.
\item The search for methods of counting the number of simplices of cylinders and subdivisions of an augmented semi-simplicial set.

\end{itemize}

In order to reach these targets we consider two different kinds of mathematical objects:

\begin{description}
  \item[(i)] Augmented semi-simplicial and co-semi-simplicial objects.

  \item[(ii)] Augmented integer sequences and matrices.

\end{description}

For further information on semi-simplicial sets, we recommend consulting the following references: \cite{EZ50}, \cite{BRS76}, \cite{F12}, \cite{DV06}, \cite{nlabsemi}, \cite{GHR2022}. To learn more about the realization of semi-simplicial sets, \cite{Ebert_2019} is a good resource. For category models related to homotopy theory of simplicial sets and topological spaces, \cite{GZ} is a relevant reference.

One noteworthy aspect of our study is that we are focusing on augmented semi-simplicial objects instead of standard semi-simplicial objects. This minor modification yields more symmetric and simplified structures and formulas, making computations much easier.

As basic combinatorial elements we consider the following ones:

\begin{description}
  \item[(bc)] Binomial coefficients (if $q>p,$ then take $\binom{p}{q}=0$):
$$\binom{p}{q}=\frac{p!}{q!(p-q)!}= \frac{p (p-1) \cdots (p-q+1)}{(p-q)!}$$
which give the number of strictly increasing
maps from the ordered set with $q$ elements $\{1< \cdots< q\}$ to the ordered set of $p$  elements $\{1< \cdots< p\}$. These numbers occur as coefficients in Newton's binomial formula
$$(a+b)^p= \sum_{q=0}^p \binom{p}{q} a^{p-q} b^q$$
\noindent and the coefficients in Pascal's triangle
$$
\begin{array}{ccccccccccccccc}
 &  & &  &  &  &1 &  & &  & &    \cr
 &  & &  &  &  1& &  1& &  & &    \cr
&   & &  &  1&  & 2&  & 1&  & &    \cr
&   & & 1 &  &  3& &  3& & 1 & &    \cr
 &  & 1&  & 4 &  & 6&  &4 &  & 1&    \cr
& 1 & & 5 & & 10 & & 10 & & 5 &  &1   \cr
\cdots & \vdots& \vdots & \vdots  & \vdots & \vdots  & \vdots & \vdots  & \vdots & \vdots  & \vdots  &\vdots  & \cdots   \cr
\end{array}
$$
\noindent which can also be represented by the matrix
$$
\begin{pmatrix}
 &0 & 1 & 2 & 3 & 4 & 5 & 6 & \cdots & \\\hline
0|& \binom{0}{0}& 0& 0& 0 & 0 & 0 & 0 & \cdots  \\
1\,|&  \binom{1}{0} &  \binom{1}{1}& 0& 0 & 0 & 0 & 0 & \cdots     \\
2|&  \binom{2}{0} &  \binom{2}{1} &  \binom{2}{2}  & 0 & 0 & 0 & 0 & \cdots    \\
3|&  \binom{3}{0} &   \binom{3}{1}  &   \binom{3}{2}   & \binom{3}{3}   & 0 & 0 & 0 & \cdots    \\
4|&   \binom{4}{0} &   \binom{4}{1}  &   \binom{4}{2}   & \binom{4}{3}   & \binom{4}{4} & 0 & 0 & \cdots    \\
5|& \binom{5}{0} &   \binom{5}{1}  &   \binom{5}{2}   & \binom{5}{3}   & \binom{5}{4} & \binom{5}{5} & 0 & \cdots    \\
6|& \vdots & \vdots & \vdots & \vdots &  \vdots  &  \vdots  &  \vdots  &  \ddots    \\
\end{pmatrix}
$$
In this work, we have considered some results about binomial numbers from \cite{GKP90}, \cite{GK90}, \cite{D76}, \cite{D77} and \cite{R68}.

\item[(cps)]
The chain-power-set numbers are studied in \cite{NS91}. They are defined as the cardinality of the set of all chains of length $k$ in the poset $\mathcal{P}(X_n)$, where $X_n=\{1,\dots,n\}$, and a chain of length $k$ has the form $\emptyset \subseteq N_0\subset N_1\subset\cdots\subset N_k\subseteq X_n$, where $N_i\neq N_{i+1}$ for $0\leq i\leq k-1$. These numbers are denoted by $a_{n,k}$ and analyzed in \cite{NS91}, which provides interesting results on recurrence formulas and relations between chain-power-set and Stirling numbers (\cite{Co74}).

In subsection \ref{subdivisionofintegers} of our work, we use the notation $\text{cad}[n]_k=a_{n+1,k}$ for these numbers. We also introduce other chain numbers denoted by $\text{cad}^+[n]_p$ and $\breve{\text{cad}}^+[n]_p$, and establish some relations between these different chain numbers and Stirling numbers, which are given in Lemma \ref{relacad} and Theorem \ref{relacadStirling2}.

\item[(s2c)] \label{Stirling} Consider a finite set $F_p=\{1, \cdots, p\}$ and $q \in \mathbb{N}$. A $q$-partition $\mathcal{P}$ of $F_p$ is a family of non-empty subsets $\mathcal{P}=\{C_1, \cdots, C_q\}$ such that $C_1\cup \cdots \cup  C_q=F_p$ and $C_i \cap C_j=\emptyset$ for $i\not=j$. Notice that there are no $q$-partitions of $F_p$ when $q>p$. For $p,q \in \mathbb{N}$, the Stirling numbers of second class are given by
$$
\Stirling2{p}{q}=
\left\{
\begin{array}{lll}
|\{ \mathcal{P}\hspace{2pt}|\hspace{2pt} \mathcal{P}  \text{ is a } q\text{-partition of } F_p \}| & \mbox{if } p \geq  q\geq 1 ,\\ [1pc]
1 & \mbox{if } p= 0, \, q= 0\\ [1pc]
0 & \mbox{ otherwise}
\end{array} \right.
$$
\end{description}

This paper introduces new construction techniques for cylinders and subdivisions, as follows: Given an augmented semi-simplicial set $X$ and a co-semi-simplicial object $Y$ (in the category of augmented semi-simplicial sets), we take the right action of $Y$ on $X$ to obtain the augmented semi-simplicial set $X \widetilde{\vartriangleright} Y$, as defined in Definition \ref{actions}.

In particular, we have introduced co-semi-simplicial objects, $\text{Cil}_0$, $\text{Cil}$, $\text{Cil}_2$ to establish three kinds of cylinders of an augmented semi-simplicial set $X$:
$$X \widetilde{\vartriangleright} \text{Cil}_0,\quad  X \widetilde{\vartriangleright} \text{Cil}, \quad  X \widetilde{\vartriangleright}\text{Cil}_2.$$
The same procedure have been used to create the barycentric subdivision of an augmented semi-simplicial set $X$. We deal with a certain co-semi-simplicial object $ \text{Sd}$, and its action on $X$ is denoted by
$$X \widetilde{\vartriangleright} \text{Sd}$$
\noindent which represents the barycentric subdivision of $X$.

As far as the algebraic context is concerned, we consider categories of augmented sequences of integers with a categorical ring structure admitting actions of certain augmented matrices. This action is determined by using the inverse matrix of the matrix $``\text{bin}"$ associated with binomial transformations. Given a finite augmented sequence $a$ and an augmented  matrix $B $, the \emph{action} of $B$ on $a$ is described in Definition \ref{algebraicaction} by the formula
$$a \widetilde{\triangleright}B:=(a \cdot \text{bin}^{-1}) \cdot B.$$
One of the main techniques used in this paper is based on the fact that the sequential cardinal functor is compatible with action operators (Theorem \ref{tildebigodottildedot}), see Theorem 5 in  \cite{GHR2022}. In other words, if $X$ is an augmented semi-simplicial set and $Z$ is an augmented co-simplicial-object of augmented semi-simplicial sets, then we have
$$|X \tilde{\vartriangleright} Z|=|X| \tilde{\triangleright} |Z|.$$
On the left side of the formula above, $ \tilde{\vartriangleright} $ stands for the action in the geometric setting, whereas on the right side $\tilde{\triangleright}$ represents the action in the algebraic one.

In subsection \ref{geometriccylinders} of this paper, we compute the corresponding augmented matrices of the co-semi-simplicial objects $\text{Cil}$, $\text{Cil}_0$, $\text{Cil}_2$:
$$\text{cil}=|\text{Cil}|, \quad \text{cil}_0=|\text{Cil}_0|, \quad \text{cil}_2=|\text{Cil}_2|.$$
Using properties of chain-power-set numbers and Stirling numbers applied to the co-semi-simplicial object $\text{Sd}$, we prove in Theorem \ref{cadtheorem} that
$$\text{cad}^+=|\text{Sd}|$$
\noindent where $\text{cad}^+$, defined in subsection \ref{subdicsequences}, is given as a slightly different version of the chain numbers.

These results allow us to count effortlessly the number of simplices of an augmented semi-simplicial finite set which is built through action operations. Regarding cylinder objects one can take the matrices
$$\breve {\text{cil}}= \text{bin}^{-1} \cdot  \text{cil}, \quad \breve {\text{cil}}_0= \text{bin}^{-1} \cdot \text{cil}_0,  \quad \breve {\text{cil}}_2= \text{bin}^{-1} \cdot  \text{cil}_2$$
and, with respect to barycentric subdivisions,
$$\breve{\text{sd}}= \breve {\text{cad}}^+= \text{bin}^{-1} \cdot \text{cad}^+$$
\noindent $\text{bin}$ being the augmented matrix of binomial coefficients.

Now, for a given augmented semi-simplicial set $X$ whose cardinal sequence is $a=|X|$, we can compute the cardinal sequences of $\text{Cil}(X), \text{Cil}_0(X),  \text{Cil}_2(X)$ and $\text{Sd}(X)$ by considering the matrix products
$$ a \cdot \breve {\text{cil}}, \quad  a \cdot \breve {\text{cil}}_0, \quad a \cdot  \breve {\text{cil}}_2,  \quad a \cdot  \breve{\text{sd}}.$$
\noindent The first arrows and columns of these matrices are given by
$$
\begin{array}{l |l cccccccc}
\breve{\text{cil}} &-1& 0 & 1 & 2 & 3 & 4 & 5 & \dots \\\hline
-1&1& 0 & 0 & 0 & 0 & 0 & 0 &  \dots&  \\
0&0& 2 & 1 & 0 & 0 & 0 & 0 &\dots& \\
1&0& 0 & 3 & 2 & 0 & 0 & 0 & \dots&  \\
2&0& 0 & 0 &4  & 3 & 0 &  0&\dots&  \\
3&0& 0 &  0& 0 & 5 &  4& 0 & \dots& \\

~~~~~~~~~~\vdots &\vdots & \vdots  &\vdots  & \vdots & \vdots  &\vdots  & \vdots   & \ddots  & \\

\end{array}
$$

$$
\begin{array}{l |l cccccccc}
\breve{\text{cil}_0}  &-1& 0 & 1 & 2 & 3 & 4 & 5 &\dots \\\hline
-1&1& 0 & 0 & 0 & 0 & 0 & 0 &  \dots&  \\
0&0& 2 & 0 & 0 & 0 & 0 & 0 & \dots& \\
1&0& 0 & 3 & 0 & 0 & 0 & 0 & \dots&  \\
2&0& 0 & 0 &4  & 0 & 0 &  0& \dots&  \\
3&0& 0 &  0& 0 & 5 &  0& 0 & \dots& \\
~~~~~~~~~~\vdots &\vdots & \vdots  &\vdots  & \vdots & \vdots  &\vdots  & \vdots   & \ddots  &  \\

\end{array}
$$

$$
\begin{array}{l |l cccccccc}
\breve{\text{cil}_2} &-1& 0 & 1 & 2 & 3 & 4 & 5 &\dots \\\hline
-1&1& 0 & 0 & 0 & 0 & 0 & 0 &  \dots&  \\
0&0& 2 & 1 & 0 & 0 & 0 & 0 &\dots& \\
1&0& 0 &   4 &  4&  1 & 0 & 0 &\dots&  \\
2&0& 0 & 0 & 8 &   12&   6&    1& \dots&  \\
3&0& 0 &  0& 0 & 16  & 32 &  24 &  \dots& \\
~~~~~~~~~~\vdots &\vdots & \vdots  &\vdots  & \vdots & \vdots  &\vdots  & \vdots  & \ddots  & \\

\end{array}
$$

$$
\begin{array}{l | ccccccccc}
 \breve{\text{cad}}^+=\ \breve{\text{sd}} & -1 & 0 & 1 & 2 & 3 & 4 & 5 &   \\ \hline
 -1&1 & 0 & 0 & 0 & 0 & 0 & 0 & \cdots \\
0& 0 & 1 & 0 & 0 & 0 & 0 & 0 &  \cdots \\
1& 0 & 1 & 2 & 0 & 0 & 0 & 0 & \cdots \\
2& 0 & 1 & 6 & 6 & 0 & 0 & 0 &  \cdots \\
3& 0 & 1 & 14 & 36 & 24 & 0 & 0 &  \cdots \\
\vdots & \vdots & \vdots & \vdots & \vdots & \vdots & \vdots & \vdots  & \vdots \\
\end{array}
$$
We note that $\Pi_{i=-1}^{\infty}\mathbb{Z}$, equipped with the product topology, is a topological group and the matrices
$\breve {\text{cil}},$ $\breve {\text{cil}}_0$ and $\breve {\text{cil}}_2$
 induce continuous group automorphisms of this group. However,  the matrix
$  \breve{\text{sd}}$ only induces a continuous group automorphism on the subgroup $\oplus_{i=-1}^{\infty}\mathbb{Z}$.
For more information on the properties of continuous homomorphisms of these Abelian topological groups, we refer the reader to \cite{Baer}, \cite{Specker} and \cite{Ferrer_2017}.

Finally, we also want to highlight the connection between the structure of cylinders and subdivisions and certain crystal molecular structures, where the  dominating features are trans-edge connected tetrahedra. For instance,  if we apply the funtor $\text{Cil}_2$ to the barycentric subdividision $\text{Sd}^2(\Gamma_+[1])$ (see Figure \ref{cildosSddos}), we obtain simplicial models for some crystal structures that are analyzed in \cite{SCHURZ2009110} and \cite{Volkova_2016}.
\begin{figure}[htbp]
\begin{center}
\includegraphics[scale=0.5]{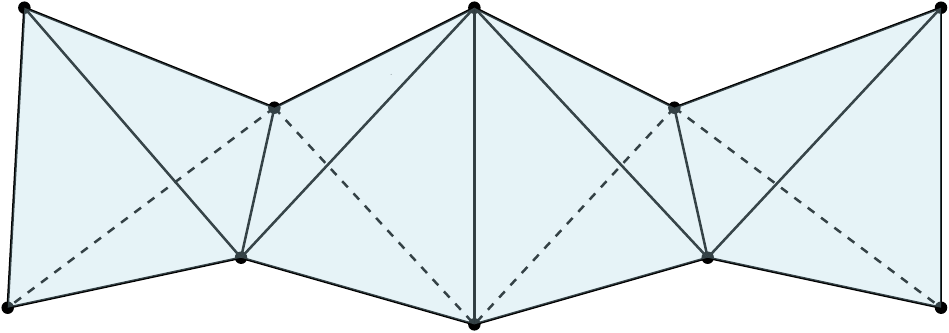}
\caption{Simplicial structure of  $ {{\rm{Cil}}}_2 \rm{Sd}^2(\Gamma_+[1])$}
\label{cildosSddos}
\end{center}
\end{figure}

\section{Augmented semi-simplicial sets and integer sequences}

In this section, we establish some notation and results regarding the categories used in this work.

\subsection{Presheaves}

We denote the category of sets as $\mathbf{Sets}$. Given a small category $\mathbf{C}$ we consider the usual functor category
$\mathbf{Sets}^{\mathbf{C}^{op}}$ which has functors $X:\mathbf{C}^{op}\rightarrow \mathbf{Sets}$ as objects, and natural transformations $f:X\rightarrow X'$ as arrows. The category $\mathbf{Sets^{C^{op}}}$ is commonly referred to as the \emph{category of presheaves on $\mathbf{C}.$ }

For a given object $c$ in $\mathbf{C}$, we can consider the presheaf $\Yoneda(c)$ on $\mathbf{C}$, defined as the contravariant Hom-functor $\Yoneda(c)(-):=\mbox{Hom}_{\mathbf{C}}(-,c)$. This construction gives rise to the well-known Yoneda embedding:
$$\Yoneda:\mathbf{C}\to \mathbf{Sets^{C^{op}}},\hspace{10pt}c\mapsto \mbox{Hom}_{\mathbf{C}}(-,c).$$
In this setting, an embedding is a full and faithfull functor. Any presheaf that is isomorphic to a presheaf of the form $\Yoneda(c)$ is called \emph{representable}. The Yoneda lemma asserts that for any presheaf
$X$, there exists a bijection between the natural transformations $\Yoneda(c)\rightarrow X$ and the elements of $X(c)$:
$$\mathbf{Nat}(\Yoneda(c),X)\stackrel{\cong }{\rightarrow}X(c),\hspace{8pt}\alpha \mapsto \alpha _c(1_c).$$

Associated with $X\in \mathbf{Sets}^{ \mathbf{C}^{op}}$, we have the so-called \emph{category of elements} of $X$, denoted by $\int_{\mathbf{C}}X$. Its objects are pairs of the form $(c,x)$ where $c$ is an object in $\mathbf{C}$ and $x\in X(c)$; a morphism $(c, x)\to (c',x')$ in $\int_{\mathbf{C}}X$ consists of a morphism
$\varphi \colon c\to c'$ in $\mathbf{C}$ such that $X(\varphi)(x')=x$. Observe that we have a projection functor:
$$\pi _X\colon \int_{\mathbf{C}} X \to \mathbf{C},\hspace{8pt}(c,x)\mapsto c.$$

The proof of the following theorem can be found in \cite{MM92}.

\medskip

\begin{teo}\label{ext-pr}
Let $Z\colon \mathbf{C} \to \mathbf{\mathcal{E}}$ be a functor from a small category $\mathbf{C}$ to a cocomplete category $\mathbf{\mathcal{E}}$. Then, the functor $\mbox{Sing}^Z\colon \mathbf{\mathcal{E}}\to \mathbf{Sets}^{\mathbf{C}^{op}}$ given by
$$\mbox{Sing}^Z(Y)(c):=\mbox{Hom}_{\mathbf{\mathcal{E}}}(Z(c),Y)$$ \noindent admits a left adjoint functor $L^Z:\mathbf{Sets}^{\mathbf{C}^{op}}\to \mathbf{\mathcal{E}}$, defined for each presheaf $X$ as
$$L^Z(X):=\mbox{colim}(\int_{\mathbf{C}}X \stackrel{\pi _X}{\longrightarrow}\mathbf{C}\stackrel{Z}\longrightarrow \mathbf{\mathcal{E}})$$
The functor $L^Z$ preserves colimits and makes commutative the following diagram
$$
\xymatrix{ \ar@{^(->}[d]_{\Yoneda} {\mathbf{C}} \ar[r]^Z &  {\mathbf{\mathcal{E}}}\\
\mathbf{Sets}^{\mathbf{C}^{op}} \ar[ru]_{L^Z} & } $$
In other words, $L^Z$ is an extension of $Z$ that preserves colimits. Moreover $L^Z$ is, up to natural isomorphism, the only extension of $Z$ preserving all colimits.
\end{teo}

\medskip \begin{obs}
Observe that, actually, we have a functor:
$$\mathbf{Sets}^{\mathbf{C}^{op}}\times \mathbf{\mathcal{E}^C}\to \mathbf{\mathcal{E}},\hspace{8pt} (X,Z)\mapsto L^Z(X).$$
\end{obs}

We note that, as a consequence of Theorem \ref{ext-pr} above with $\mathbf{\mathcal{E}}=\mathbf{Sets}^{\mathbf{C}^{op}}$, certain induced functors arise, which are called action functors:

\medskip
\begin{defi}\label{actions} The \emph{action functors} are the following ones:
\begin{enumerate}
\item $(-)\widetilde{\vartriangleright} (-) \colon  \mathbf{Sets}^{\mathbf{C}^{op}}  \times (\mathbf{Sets}^{\mathbf{C}^{op}})^{\mathbf{C}} \to \mathbf{Sets}^{\mathbf{C}^{op}}$ $$X \widetilde{\vartriangleright} Y:=L^Y(X).$$

\item $(-)\widetilde{\vartriangleright} (-) \colon  (\mathbf{Sets}^{\mathbf{C}^{op}})^{\mathbf{C}}\times (\mathbf{Sets}^{\mathbf{C}^{op}})^{\mathbf{C}} \to (\mathbf{Sets}^{\mathbf{C}^{op}})^{\mathbf{C}}$ $$(Y \widetilde{\vartriangleright} Z)(c):=Y(c)\widetilde{\vartriangleright} Z=L^Z(Y(c)).$$
\end{enumerate}

Given $X\in \mathbf{Sets}^{\mathbf{C}^{op}} $ and $Y, Z \in (\mathbf{Sets}^{\mathbf{C}^{op}})^{\mathbf{C}} $ we say that $X \widetilde{\vartriangleright} Y \in \mathbf{Sets}^{\mathbf{C}^{op}}$ is the (right) action
of $Y$ on $X$, and similarly, $Y \widetilde{\vartriangleright} Z \in (\mathbf{Sets}^{\mathbf{C}^{op}})^{\mathbf{C}} $ is the (right) action of $Z$ on $Y$.
 \end{defi}

\medskip
 \begin{obs} If $X\in \mathbf{Sets}^{\mathbf{C}^{op}} $ and $Y, Z \in  (\mathbf{Sets}^{\mathbf{C}^{op}})^{\mathbf{C}}$, then it follows that
$$(X \widetilde{\vartriangleright} Y)  \widetilde{\vartriangleright}  Z\cong X \widetilde{\vartriangleright} (Y  \widetilde{\vartriangleright}  Z).$$ This formula inspires the name given: ``action functors''.
 \end{obs}

\subsection{Augmented semi-simplicial sets}

Now consider the small category $\mathbf{\Gamma}$, whose objects are all non-empty totally ordered sets $[p]=\{0<\cdots<p\}$ for $p\geq 0$, and whose morphisms are strictly increasing maps $[p]\rightarrow [q].$ We can add the empty set $\emptyset =[-1]$ to this category together with all the strictly increasing maps. The resulting augmented category will be denoted by $\mathbf{\Gamma _+}$.

We can consider $\mathbf{Sets}^{\mathbf{\Gamma}_+^{op}}$ and $\mathbf{Sets}^{\mathbf{\Gamma}_+}$, which are the presheaf categories on $\mathbf{\Gamma} _+$ and $\mathbf{\Gamma} _+^{op}$, respectively. Taking $\mathbf{\mathcal{E}}=\mathbf{Sets}^{\mathbf{\Gamma}_+^{op}}$, $\mathbf{C}=\mathbf{\Gamma} ^+$, and using the notation $X\widetilde{\vartriangleright}Y$ for $L^Y(X)$ given in Definition \ref{actions} we obtain the following action functors:
 $$\mathbf{Sets}^{\mathbf{\Gamma}_+^{op}}\times (\mathbf{Sets}^{\mathbf{\Gamma}_+^{op}})^{\mathbf{\Gamma}_+}\to \mathbf{Sets}^{\mathbf{\Gamma}_+^{op}},\hspace{8pt}(X,Y)\mapsto X\widetilde{\vartriangleright}Y,$$
 $$(\mathbf{Sets}^{\mathbf{\Gamma}_+^{op}})^{\mathbf{\Gamma}_+}\times (\mathbf{Sets}^{\mathbf{\Gamma}_+^{op}})^{\mathbf{\Gamma}_+}\to (\mathbf{Sets}^{\mathbf{\Gamma}_+^{op}})^{\mathbf{\Gamma}_+},\hspace{8pt}(Y,Z)\mapsto Y\widetilde{\vartriangleright}Z.
$$

We obviously have the corresponding Yoneda embeddings:
$$\Yoneda:\mathbf{\Gamma}_+\to \mathbf{Sets}^{\mathbf{\Gamma}_+^{op}}, \quad \Yoneda^{op}:\mathbf{\Gamma}_+^{op} \to \mathbf{Sets}^{\mathbf{\Gamma}_+}.$$
For any object $[n]$ in $\mathbf{\Gamma}_+$ we denote by $\mathbf{\Gamma}_+[n]$ the representable presheaf $$\mathbf{\Gamma}_+[n]:=\Yoneda([n])=\mbox{Hom}_{\mathbf{\Gamma}_+}(-,[n])\in \mathbf{Sets}^{\mathbf{\Gamma}_+^{op}}.$$ Analogously, we consider the notation
$$\mathbf{\Gamma}_+^{op}[n]:=\Yoneda^{op}([n])=\mbox{Hom}_{\mathbf{\Gamma}_+}([n],-)\in \mathbf{Sets}^{\mathbf{\Gamma}_+}.$$

\begin{defi} Any presheaf on $\mathbf{\Gamma}_+$ $$X:\mathbf{\Gamma}_+^{op}\to \mathbf{Sets}$$ \noindent will be called \emph{augmented semi-simplicial set} or $\mathbf{\Gamma}_+^{op}$-\emph{set}. Analogously, any presheaf on $\mathbf{\Gamma}_+^{op}$, $Z:\mathbf{\Gamma}_+\to
\mathbf{Sets}$, will be called \emph{augmented co-semi-simplicial set} or $\mathbf{\Gamma}_+$-\emph{set}.
\end{defi}

Giving an augmented semi-simplicial set
$X\in \mathbf{Sets}^{\mathbf{\Gamma}_+^{op}}$
 is equivalent to giving a collection of sets $\{X_n\}_{n\geq -1}$ together with
a collection of set maps $d_i^n:X_n\rightarrow X_{n-1}$ ($n\geq 0$ and $0\leq i\leq n$) satisfying
$$d_i^n\circ d_j^{n+1}=d_{j-1}^n\circ d_i^{n+1},\hspace{4pt}\mbox{if}\hspace{2pt}i<j.$$

Moreover, giving an arrow $f:X\rightarrow \bar X$ in $\mathbf{Sets}^{\mathbf{\Gamma}_+^{op}}$ (i.e., a natural transformation) is equivalent to giving a collection of set maps $\{f_n:X_n\rightarrow \bar X_n\}_{n\geq -1}$ such that
$$f_{n-1}\circ d_i^n={\bar d}_i^{n-1}\circ f_n,$$ for $n\geq 0$ and $0\leq i\leq n.$

There is a similar description for the category of augmented co-semi-simplicial sets by just reversing the arrows in the representation above.

We can also consider the subcategory of finite sets $\mathbf{Sets}_{\rm{fin}}$ and the corresponding category of presheaves
$(\mathbf{Sets}_{\rm{fin}})^{\mathbf{\Gamma}_+^{op}}$.

An augmented semi-simplicial set $X\in\mathbf{Sets}^{\mathbf{\Gamma}_+^{op}}$ is said to have \emph{finite dimension} if there is $k\in \mathbb{N}_+$ such that $X_i=\emptyset$ for every $i > k$. Given a non-empty finite dimensional semi-simplicial set $X$, we denote $\text{dim}(X):=\min\{k \hspace{3pt}|\hspace{3pt}X_i=\emptyset \hspace{3pt}\mbox{for all}\hspace{3pt} i > k, i  \in \mathbb{N}_+\}.$
For the empty semi-simplicial set, we set $\text{dim}(\emptyset):=-\infty $

The subcategory of finite dimensional semi-simplicial sets is denoted by $(\mathbf{Sets}^{\mathbf{\Gamma}_+^{op}})_{\text{fd}}$.
An augmented semi-simplicial set $X\in\mathbf{Sets}^{\mathbf{\Gamma}_+^{op}}$ is said to be \emph{finite} if it has finite dimension and for every $i \in \mathbb{N}_+$, $X_i$ is a finite set. The subcategory of finite semi-simplicial sets is denoted by $(\mathbf{Sets}^{\mathbf{\Gamma}_+^{op}})_{\text{fin}}$.

The join of any pair of $\mathbf{\Gamma}_+^{op}$-sets, $X,Y$, is given by the formula:
$$(X\boxplus Y)_m:=\bigsqcup_{p+q=m-1} X_p \times Y_q$$ where $p$ and $q$ are integers greater than or equal to $-1.$
In particular, we have $(X\boxplus Y)_{-1}=X_{-1} \times Y_{-1}$, $(X\boxplus Y)_{0}=(X_{-1} \times Y_{0})\sqcup(X_{0} \times Y_{-1})$,
$(X\boxplus Y)_{1}=(X_{-1} \times Y_{1})\sqcup (X_{0} \times Y_{0})\sqcup (X_{1} \times Y_{-1})$ and so on.
Additionally, the operators $d_i^{X\boxplus Y}$  of $X\boxplus Y$ are defined naturally from the operators $d_k^{X}$ and $d_l^{Y}$ of $X$ and $Y$ respectively. The definition of $\boxplus $ on morphisms is straightforward. This gives the \emph{join functor}
$$\boxplus :\mathbf{Sets}^{\mathbf{\Gamma}_+^{op}} \times \mathbf{Sets}^{\mathbf{\Gamma}_+^{op}} \to \mathbf{Sets}^{\mathbf{\Gamma}_+^{op}} $$

\medskip
For any pair $X,Y \in \mathbf{Sets}^{\mathbf{\Gamma}_+^{op}}$ we have that the functors
$$X\boxplus (-),\hspace{4pt}(-)\boxplus Y:\mathbf{Sets}^{\mathbf{\Gamma}_+^{op}}\to \mathbf{Sets}^{\mathbf{\Gamma}_+^{op}}$$ \noindent preserve colimits (see \cite{GHR2022}).
Furthermore, if $X, Y, Z \in \mathbf{Sets}^{\mathbf{\Gamma}_+^{op}}$, then there exist canonical natural isomorphisms
$$X\boxplus(Y\boxplus Z) \cong (X \boxplus Y)\boxplus Z,$$
$${\Gamma_+}[-1] \boxplus X\cong X \cong X \boxplus {\Gamma_+}[-1], $$
$$X  \boxplus Y \cong Y  \boxplus X.$$
Moreover, it is shown in \cite{GHR2022} that for any $n,m\in \mathbb{N}_+$, the following isomorphism holds true
$${\Gamma_+}[n]\boxplus {\Gamma_+}[m]\cong {\Gamma_+}[n+m+1].$$

\medskip
The category $\mathbf{Sets}^{\mathbf{\Gamma}_+^{op}}$ together with the join functor $\boxplus$ and the unit object ${\Gamma_+}[-1]$ forms a symmetric monoidal category \cite{GHR2022}. Furthermore, considering the coproduct
(the ordinal sum) in the category $\mathbf{\Gamma}_+$
$$[p]\sqcup [q]:=[p+q+1]$$ we obtain a symmetric monoidal category structure on $\mathbf{\Gamma}_+$ having $[-1]$ as a unit object.
It is also shown in \cite{GHR2022} that the Yoneda embedding, $\Yoneda : (\mathbf{\Gamma_+}, \sqcup, [-1]) \to (\mathbf{Sets}^{\mathbf{\Gamma}_+^{op}},\boxplus,\Gamma_+[-1]),$ is monoidal. Moreover,  $((\mathbf{Sets}^{\mathbf{\Gamma}_+^{op}})_{\text{fin}},\boxplus,\Gamma_+[-1])$ is a monoidal subcategory of $((\mathbf{Sets}_{\text{fin}})^{{\mathbf{\Gamma}_+^{op}}},\boxplus,\Gamma_+[-1])$ and the latter is a monoidal subcategory of $(\mathbf{Sets}^{\mathbf{\Gamma}_+^{op}},\boxplus,\Gamma_+[-1])$.

\bigskip
We have the \emph{left-cone functor} $\mbox{Con}_l:={\Gamma_+}[0]\boxplus (-)$ and the \emph{right-cone functor}
$\mbox{Con}_r:=(-)\boxplus {\Gamma_+}[0]$:
$$\mbox{Con}_l, \mbox{Con}_r:\mathbf{Sets}^{\mathbf{\Gamma}_+^{op}} \to \mathbf{Sets}^{\mathbf{\Gamma}_+^{op}}$$
$$\mbox{Con}_l(X)={\Gamma_+}[0]\boxplus X,$$
$$\mbox{Con}_r(X)= X\boxplus {\Gamma_+}[0].$$

These functors satisfy
$\mbox{Con}_l  ({\Gamma_+}[k])={\Gamma_+}[k+1]= \mbox{Con}_r({\Gamma_+}[k]),$
for all $k\geq -1.$

\bigskip

Recall that, using the notation $X\widetilde{\vartriangleright}Y$ defined in Definition \ref{actions}, we have the action functors:
$$\mathbf{Sets}^{\mathbf{\Gamma}_+^{op}}\times (\mathbf{Sets}^{\mathbf{\Gamma}_+^{op}})^{\mathbf{\Gamma}_+}\to \mathbf{Sets}^{\mathbf{\Gamma}_+^{op}},\hspace{8pt}(X,Y)\mapsto X\widetilde{\vartriangleright}Y$$
$$(\mathbf{Sets}^{\mathbf{\Gamma}_+^{op}})^{\mathbf{\Gamma}_+}\times (\mathbf{Sets}^{\mathbf{\Gamma}_+^{op}})^{\mathbf{\Gamma}_+}\to (\mathbf{Sets}^{\mathbf{\Gamma}_+^{op}})^{\mathbf{\Gamma}_+}, \hspace{8pt}(Y,Z)\mapsto Y\widetilde{\vartriangleright}Z$$

Here, we take $\mathbf{\mathcal{E}}=\mathbf{Sets}^{\mathbf{\Gamma}_+^{op}}$ and $\mathbf{C}=\mathbf{\Gamma}^+$.
\medskip

The following result is also proved in \cite{GHR2022}. Recall that we are using the particular notation $(-) \widetilde{\vartriangleright} Z$ for the construction $L^Z(-)$:
\medskip
\begin{prop}\label{superj} Let $Z:(\mathbf{\Gamma}_+, \sqcup,[-1]) \to (\mathbf{Sets}^{\mathbf{\Gamma}_+^{op}},\boxplus,\Gamma_+[-1])$ be a monoidal functor. Then the colimit-preserving functor $(-)\widetilde{\vartriangleright} Z:\mathbf{Sets}^{\mathbf{\Gamma}_+^{op}}\to \mathbf{Sets}^{\mathbf{\Gamma}_+^{op}}$, which makes the following diagram commute
$$
\xymatrix{ \ar@{^(->}[d]_{\Yoneda} {\mathbf{\Gamma_+}} \ar[r]^Z & \mathbf{Sets}^{\mathbf{\Gamma}_+^{op}}\\
\mathbf{Sets}^{\mathbf{\Gamma}_+^{op}}\ar[ru]_{(-)\widetilde{\vartriangleright} Z} & } $$
\noindent is monoidal. In particular, for all $X,Y \in \mathbf{Sets}^{\mathbf{\Gamma}_+^{op}}$,
we have
$$(X\boxplus Y)\widetilde{\vartriangleright}Z\cong (X\widetilde{\vartriangleright}Z)\boxplus (Y\widetilde{\vartriangleright}Z).$$
Moreover, up to isomorphism, it is the unique colimit-preserving functor making this diagram commute.
\end{prop}

\medskip
\begin{obs}
We also obtain:

\begin{itemize}
 \item[(i)] If $Z:(\mathbf{\Gamma_+}, \sqcup,[-1]) \to ((\mathbf{Sets}_{\rm fin})^{\mathbf{\Gamma}_+^{op}},\boxplus,\Gamma_+[-1])$ is a monoidal functor, then $(-)\widetilde{\vartriangleright} Z: (\mathbf{Sets}_{\rm fin})^{\mathbf{\Gamma}_+^{op}}\rightarrow (\mathbf{Sets}_{\rm fin})^{\mathbf{\Gamma}_+^{op}}$ is a monoidal functor.
  \item[(ii)] If  $Z:(\mathbf{\Gamma_+}, \sqcup,[-1]) \to ((\mathbf{Sets}^{\mathbf{\Gamma}_+^{op}})_{\rm fin},\boxplus,\Gamma_+[-1])$ is a monoidal functor,
then $(-)\widetilde{\vartriangleright} Z: (\mathbf{Sets}^{\mathbf{\Gamma}_+^{op}})_{\rm fin}\rightarrow (\mathbf{Sets}^{\mathbf{\Gamma}_+^{op}})_{\rm fin}$ is a monoidal functor.

\end{itemize}
\end{obs}

\subsection{Augmented integer sequences and matrices}

We study the second kind of mathematical objects in this work: augmented integer sequences (and matrices).
The set of integer numbers $\mathbb{Z}$ admits the structure of a discrete category. However, we can also consider it as a groupoid ${\mathbf{Z}}$ where the cardinal $|\text{Hom}_{{\mathbf{Z}}}(n,m)|=1$ and the unique morphism from $n$ to $m$ is denoted by $m-n\colon n \to m$, for every pair of integers $n,m$. The sum of integers can be easily extended to a functor
$$+:\mathbf{Z}\times \mathbf{Z}\rightarrow \mathbf{Z},\hspace{4pt}(n,m)\mapsto n+m$$ Taking $+$ as a tensor product and $0$ as a unit object, it is immediate to check that $({\bf{Z}}, +, 0)$ has the structure of a strict symmetric monoidal category and also of a strict categorical group.

Given any small category $\mathbf{J}$, the functor category $\mathbf{Z^J}$ has an induced strict symmetric monoidal category structure (in addition, it is a strict symmetric categorical group).  We consider $(\mathbf{Z^J})_{\text{fin}}$ the full subcategory of $\mathbf{Z^J}$ consisting of functors $c:\mathbf{J}\to \mathbf{Z}$ such that there exists a finite set of objects $F_c$ satisfying that $c(j)=0$, for all $j\in \mathbf{J}\setminus F_c$.

Now, if $\mathbb{N}$ denotes the set of natural numbers ($0$ is also included as a natural number), take the discrete category $\mathbb{N}_+=\mathbb{N}\cup \{-1\}$. This category can be considered as a non-full subcategory of both $\mathbf{\Gamma}_+$ and $\mathbf{\Gamma}_+^{op}$ through the (inclusion) functor $n\to [n]$. Observe that the category $\mathbb{N}_+$ is self-dual, that is, $\mathbb{N}_+=\mathbb{N}_+^{op}$.

Given a functor $c\in (\mathbf {Z}^{\mathbb{N}_+^{op}})_{\text{fin}}$ with $c\not=0$, the \emph{dimension of $c$} is the integer
$\text{dim}(c):=\min\{k \hspace{3pt}|\hspace{3pt}c_i=0 \hspace{3pt}\mbox{for all}\hspace{3pt} i > k,  i  \in \mathbb{N}_+\}.$
If $c=0$, then we set $\text{dim}(c)=-\infty .$

The Yoneda embedding
$\Yoneda:{\mathbb{N}_+}\rightarrow \mathbf{Sets}^{{\mathbb{N}_+^{op}}}$ associated to $\mathbb{N}_+$ (note that this functor is a restriction of the Yoneda functor $\Yoneda:{{\bf \Gamma}_+}\rightarrow \mathbf{Sets}^{{\bf \Gamma}_+^{op}}$) induces, by applying the cardinal operator to the corresponding hom-sets, a functor $\yoneda:{\mathbb{N}_+}\rightarrow \mathbf{Z}^{\mathbb{N}_+^{op}}$
satisfying $\yoneda(n)(j)=\delta_{n,j}$, for all $n, j\in \mathbb{N}_+.$ Here $\delta_{n,j}$ denotes the Kronecker delta
$$\delta_{n,j}=\begin{cases}
1, & n=j \\
0, & n\neq j
\end{cases}
$$
Therefore, $\yoneda(n)\in (\mathbf{Z}^{\mathbb{N}_+^{op}})_{\text{fin}}$ and $\mbox{dim}(\yoneda(n))=n,$ for all $n\in \mathbb{N}_+$. This way, we actually have a functor
$\yoneda:{\mathbb{N}_+}\rightarrow (\mathbf{Z}^{\mathbb{N}_+^{op}})_{\text{fin}}.$

\medskip
\begin{defi} An \emph{augmented integer sequence} is a functor $$a:\mathbb{N}_+^{op} \to {\mathbb{Z}}.$$
 We will denote an augmented sequence $a$ by means of a row matrix
$$a=(a_{-1} \hspace{5pt} a_0 \hspace{5pt} a_1\hspace{5pt} a_2\hspace{5pt} \cdots ).$$
However, in some cases, this row matrix will be denoted by (using commas): $a=(a_{-1}, a_0, a_1, a_2, \cdots ).$

Analogously, an \emph{augmented integer co-sequence} is a functor $b:\mathbb{N}_+ \to  {\mathbb{Z}}.$
We will denote an augmented co-sequence
$b$ by means of a column matrix
$$b=\begin{pmatrix}
  b_{-1}       \\
  b_0  \\
  b_1\\
   \vdots\\
\end{pmatrix}$$ \noindent or $b=(b_{-1} \, b_0 \, b_1\,\cdots )^T$ , where $T$ denotes the transposition operator.
\end{defi}

\medskip
\begin{defi} An \emph{augmented integer matrix} is a functor $U:\mathbb{N}_+ \times  \mathbb{N}_+^{op} \to {\mathbb{Z}}$.
An augmented matrix $U$ will be denoted by its usual form
$$U=\begin{pmatrix}
   U_{-1, -1}     &   U_{-1, 0} &   U_{-1, 1} & \cdots \\
  U_{0, -1}     &   U_{0, 0} &   U_{0, 1} & \cdots  \\
  U_{1, -1}     &   U_{1, 0} &   U_{1, 1} & \cdots  \\
   \vdots    &      \vdots &      \vdots & \ddots
\end{pmatrix}$$
\end{defi}

\medskip
Note that, as there are isomorphisms of categories $(\mathbb{Z}^{\mathbb{N}_+ }  )^{\mathbb{N}_+^{op}}   \cong  \mathbb{Z}^{\mathbb{N}_+ \times  \mathbb{N}_+^{op}}  \cong \mathbb{Z}^{\mathbb{N}_+^{op} \times  \mathbb{N}_+} \cong  (\mathbb{Z}^{\mathbb{N}_+ ^{op}} )^{\mathbb{N}_+},$
any augmented  matrix $U:\mathbb{N}_+ \times  \mathbb{N}_+^{op} \to {\mathbb{Z}}$ may be considered as an object in any of the categories above.

\medskip

If $a\in (\mathbf {Z}^{{\mathbb{N}_+^{op}}})_{\text{fin}}$ and $b:{\mathbb{N}_+} \to \mathbf{C}$ is any functor, then we have that they are of the form
$a=(a_{-1},a_0,a_1,\cdots a_n, 0,0,0,\cdots)\hspace{6pt}\mbox{and}\hspace{6pt}b=(b_{-1},b_0,b_1,b_2,b_3,\cdots )^T,$ respectively.
Therefore, $a\cdot b=\sum_{i=-1}^{+\infty} a_i b_i$ is well defined and we have an induced bifunctor
$$(-){\cdot} (-):(\mathbf {Z}^{{\mathbb{N}_+^{op}}})_{\text{fin}}\times \mathbf {Z}^{{\mathbb{N}_+}} \to \mathbf {Z},\hspace{10pt} (a,b)\mapsto a\cdot b.$$
Similarly, we have a bifunctor
$$(-){\cdot} (-):\mathbf {Z}^{{\mathbb{N}_+^{op}}}\times (\mathbf {Z}^{{\mathbb{N}_+}})_{\text{fin}} \to \mathbf {Z},\hspace{10pt} (c,d)\mapsto c\cdot d.$$

Taking into account the transposition isomorphism $(-)^T:\mathbf{Z}^{{\mathbb{N}_+^{op}}}\stackrel{\cong }{\rightarrow }\mathbf {Z}^{{\mathbb{N}_+}},$ the composition induced by the identity on the first variable and the transposition on the second variable induces the scalar (or inner) product:
$$\langle -, -\rangle :(\mathbf {Z}^{{\mathbb{N}_+^{op}}})_{\text{fin}} \times (\mathbf {Z}^{{\mathbb{N}_+^{op}}})_{\text{fin}}\to \mathbf {Z}.$$
Namely, if $a=(a_{-1},a_0,a_1,\cdots a_n, 0\cdots)$ and $b=(b_{-1},b_0,b_1,\cdots b_m, 0\cdots)$, then
$$\langle a, b\rangle=a\cdot b^T= \sum_{i=-1}^{\min\{n,m\}} a_i b_i.$$

One has the following canonical extended bifunctors of the dot-product:
$$(-)\cdot (-):(\mathbf{Z}^{{\mathbb{N}_+^{op}}})_{\text{fin}} \times ({\mathbf{Z}^{{\mathbb{N}_+^{op}}}})^{{\mathbb{N}_+}}\to \mathbf{Z}^{{\mathbb{N}_+^{op}}},\hspace{8pt} (a,B)\mapsto a\cdot B,$$
$$(-)\cdot (-):\mathbf{Z}^{{\mathbb{N}_+^{op}}} \times (({\mathbf{Z}^{{\mathbb{N}_+}}})_{\text{fin}})^{{\mathbb{N}_+^{op}}}\to \mathbf{Z}^{{\mathbb{N}_+^{op}}},\hspace{8pt} (a,B)\mapsto a\cdot B,$$
$$(-)\cdot (-):(\mathbf{Z}^{{\mathbb{N}_+^{op}}})^{{\mathbb{N}_+}} \times
({\mathbf{Z}^{{\mathbb{N}_+}}})_{\text{fin}}\to  \mathbf{Z}^{{\mathbb{N}_+}},\hspace{8pt} (A,b)\mapsto A\cdot b,$$
$$(-)\cdot (-):((\mathbf{Z}^{{\mathbb{N}_+^{op}}})_{\text{fin}})^{{\mathbb{N}_+}} \times
{\mathbf{Z}^{{\mathbb{N}_+}}}\to  \mathbf{Z}^{{\mathbb{N}_+}},\hspace{8pt} (A,b)\mapsto A\cdot b,$$

$$(a\cdot B)_{j}=\sum_{k \in \mathbb{N}_+}a_{k} B_{k,j}, \quad (A\cdot b)_{i}=\sum_{k \in \mathbb{N}_+} A_{i,k} b_{k}.$$

$$(-)\cdot (-):((\mathbf{Z}^{{\mathbb{N}_+^{op}}})_{\text{fin}} )^{{\mathbb{N}_+}} \times  ({\mathbf{Z}^{{\mathbb{N}_+}}})^{{\mathbb{N}_+^{op}}}\to \left (\mathbf{Z}^{{\mathbb{N}_+}}\right )^{{\mathbb{N}_+^{op}}},\hspace{8pt} (A,B)\mapsto A\cdot B,$$
$$(-)\cdot (-):(\mathbf{Z}^{{\mathbb{N}_+^{op}}})^{{\mathbb{N}_+}} \times
 \left (({\mathbf{Z}^{{\mathbb{N}_+}}})_{\text{fin}}\right)^{{\mathbb{N}_+^{op}}}\to (\mathbf{Z}^{{\mathbb{N}_+^{op}}} )^{{\mathbb{N}_+}},\hspace{8pt} (A,B)\mapsto A\cdot B,$$

$$(A\cdot B)_{i,j}=\sum_{k \in \mathbb{N}_+}A_{i,k} B_{k,j}.$$

\medskip
The category ${\mathbf{Z}}^{{\mathbb{N}_+^{op}}}$ has the structure of a ring $( {\mathbf{Z}}^{{\mathbb{N}_+^{op}}}, +,\times)$, which is induced by the ring structure of $({\mathbf{Z}},+,\times)$ by pointwise operation.
For $a, b \in \mathbf{Z}$ we will also use the notation $a\times b = a b$.
However, we may consider a new symmetric monoidal structure on ${\mathbf{Z}}^{{\mathbb{N}_+^{op}}}$. This is given by the join product:

\medskip

The \emph{join product} of $a, b \in{\bf{Z}}^{{\mathbb{N}_+^{op}}}$, denoted as $a\boxplus b$, is given by the following formula:
$$(a\boxplus b)_m:=\sum_{p+q=m-1} a_p b_q,  \quad p, q \in {\mathbb{N}_+}.$$

In this case the unit object is given as ${\bf 1}_{-1}$ where $({\mathbf{1}}_{-1})_i=\delta_{-1, i}$ is the Kronecker delta. In this work, for $k \in \mathbb{N}_+$,
${\mathbf{1}}_{k}$ will denote the augmented sequence given by $({\mathbf{1}}_{k})_i=\delta_{k, i}$.

The category  ${\mathbf{Z}}^{{\mathbb{N}_+^{op}}}$ equipped with the join product $\boxplus$ and the unit object ${\bf 1}_{-1}$  has the structure of a strict symmetric monoidal category (\cite{GHR2022}).  Moreover, the induced functor $\yoneda\colon (\mathbb{N}_+, \sqcup, [-1]) \to ({\mathbf{Z}}^{{\mathbb{N}_+^{op}}},\boxplus, {\bf 1}_{-1}) $ is monoidal.

Obviously, $(({\mathbf{Z}}^{{\mathbb{N}_+^{op}}})_{\text{fin}},\boxplus, {\bf 1}_{-1}) $ is a monoidal subcategory of  $({\mathbf{Z}}^{{\mathbb{N}_+^{op}}},\boxplus, {\bf 1}_{-1}) $.

If $a, b \in {\bf{Z}}^{{\mathbb{N}_+^{op}}}$ are fixed, then we easily obtain functors
$$a\boxplus (-) ,  (-) \boxplus b \colon  {\bf{Z}}^{{\mathbb{N}_+^{op}}}\to {\bf{Z}}^{{\mathbb{N}_+^{op}}}$$ \noindent given by $c\mapsto a\boxplus c$ and $c\mapsto c\boxplus b$, respectively.

\bigskip
Now, for any $k\in \mathbb{Z}$, we define an operator in the category ${\bf{Z}}^{{\mathbb{N}_+^{op}}}$ (actually, a functor)
$$\mbox{D}_k:{\mathbf{Z}}^{{\mathbb{N}_+^{op}}} \to {\mathbf{Z}}^{{\mathbb{N}_+^{op}}}, \quad b \mapsto \mbox{D}_k(b)$$ as follows:
given $b\in{\mathbf{Z}}^{{\mathbb{N}_+^{op}}}$ and $i\in {\mathbb{N}_+}$,

\begin{itemize}
\item[] If $k\geq 0$, then $(\mbox{D}_k(b))_i=b_{i+k}$,

\item[] If $k\leq 0$, then:
$$(\mbox{D}_k(b))_i=
\begin{cases}
b_{i+k},    & \mbox{if}\hspace{4pt} i+k\geq -1, \\
0,   &  \mbox{if}\hspace{4pt}  i+k<-1.
\end{cases}
$$
\end{itemize}
\medskip
\begin{defi}
For any given $b\in {\mathbf{Z}}^{{\mathbb{N}_+^{op}}}$, we define $R(b)\in \left (({\bf{Z}}^{{\mathbb{N}_+}})_{\text{fin}}\right ) ^{\mathbb{N}_+^{op}}$, the \emph{shifting} of $b$, by the formula
$(R(b))_{i}=\mbox{D}_{-(i+1)}(b),$ for all $i\in {\mathbb{N}_+}$. We may also see it as the matrix
$$R(b)=\begin{pmatrix} \mbox{D}_{0}(b)\\ \mbox{D}_{-1}(b)\\ \mbox{D}_{-2}(b)\\ \vdots\end{pmatrix}
=\begin{pmatrix}
 b_{-1}  &  b_0 & b_1 & b_2 & \cdots \\
   0 &  b_{-1} &  b_0 &  b_1 & \cdots  \\
  0  &  0 &   b_{-1} & b_0 & \cdots  \\
   \vdots    &      \vdots &      \vdots & \ddots
\end{pmatrix}$$
\end{defi}

This construction naturally gives rise to a functor
$R:{\mathbf{Z}}^{{\mathbb{N}_+^{op}}} \to \left (({\bf{Z}}^{{\mathbb{N}_+}})_{\text{fin}}\right ) ^{\mathbb{N}_+^{op}}$
which satisfies the equalities
$$a\boxplus b=a\cdot R(b)=b\cdot R(a)$$ \noindent for all $a, b \in {\mathbf{Z}}^{{\mathbb{N}_+^{op}}}$.

Now we present an interesting construction: the cone of a sequence.
Indeed, the \emph{cone} of $c\in {\mathbf{Z}}^{{\mathbb{N}_+^{op}}}$ is the sequence $\mbox{con}(c)\in {\mathbf{Z}}^{{\mathbb{N}_+^{op}}}$ defined as $\mbox{con}(c):=c+\text{D}_{-1}(c)$. That is to say,
$$
\mbox{con}(c)_{i}=
\begin{cases}
c_i+c_{i-1},   & \mbox{if}\hspace{4pt} i\geq 0 \\
c_{-1},   & \mbox{if}\hspace{4pt} i=-1.
\end{cases}
$$

We obtain the \emph{cone functor} $\mbox{con}:{\mathbf{Z}}^{{\mathbb{N}_+^{op}}}\to {\mathbf{Z}}^{{\mathbb{N}_+^{op}}}.$ Considering $\mathbf{c}=\mbox{con}( {\bf 1}_{-1})= {\bf 1}_{-1}+{\bf 1}_0 \in {\mathbf{Z}}^{{\mathbb{N}_+^{op}}}$ the cone functor is related to the join and the dot product through the following formula: $$\mbox{con}(b)= {\mathbf{c}}\boxplus b=b\boxplus {\mathbf{c}}=b\cdot R({\mathbf{c}}).$$

To finish this section we consider actions of sequences and matrices.
We first establish the augmented binomial matrix $\text{bin} \in ( ({{\mathbf Z}}^{{\mathbb{N}}_+^{op}})_{\text{fin}})^{\mathbb{N}_+}$ defined as $$\text{bin}_{i,j}=\binom{i+1}{j+1}, \quad i, j \in \mathbb{N}_+$$ and its inverse matriz $\text{bin}^{-1} \in ( ({{\mathbf Z}}^{{\mathbb{N}}_+^{op}})_{\text{fin}})^{\mathbb{N}_+}$ given by
$$\text{bin}_{i,j}^{-1} = (-1)^{i-j} \binom{i+1}{j+1}, \quad i, j \in \mathbb{N}_+.$$

\begin{defi} \label{algebraicaction} Given a sequence $a \in ({{\bf Z}}^{{\mathbb{N}}_+^{op}})_{\text{fin}} $ and a matrix $B \in ({{\mathbf Z}}^{{\mathbb{N}}_+^{op}})^{\mathbb{N}_+}$,
it is defined the \emph{action} of $B$ on $a$ by the formula
$$a \widetilde{\triangleright}B:=(a \cdot \text{bin}^{-1}) \cdot B.$$
The resulting sequence $a \widetilde{\triangleright}B$ is also said to be the \emph{tilde-triangle product} of $a$ and $B$.
 This construction gives rise to an action  functor
$$(-)\widetilde{\triangleright }(-):({{\mathbf Z}}^{{\mathbb{N}}_+^{op}})_{\text{fin}}  \times ({{\mathbf Z}}^{{\mathbb{N}}_+^{op}})^{\mathbb{N}_+}
\to  {{\mathbf Z}}^{{\mathbb{N}}_+^{op}} .$$\end{defi}
We point out that we also have the identity
$a \widetilde{\triangleright}\text{bin} =(a \cdot \text{bin}^{-1}) \cdot \text{bin}=a.$

\begin{obs}
There is an obvious extension of the dot product when $a\in {\mathbf{Z}}^{{\mathbb{N}_+^{op}}}$ is a general sequence and $C\in  (({\mathbf{Z}^{{\mathbb{N}_+}}})_{\text{fin}})^{{\mathbb{N}^{op}_+}}$:
$$(-)\cdot (-):\mathbf{Z}^{{\mathbb{N}_+^{op}}} \times   (({\mathbf{Z}^{{\mathbb{N}_+}}})_{\text{fin}})^{{\mathbb{N}^{op}_+}}\to \mathbf{Z}^{{\mathbb{N}_+^{op}}},\hspace{8pt} (a,C)\mapsto a\cdot C$$
Note that $C\in (({\mathbf{Z}^{{\mathbb{N}_+}}})_{\text{fin}})^{{\mathbb{N}^{op}_+}}$ if, and only if, the columns of the matrix $C$ are eventually constant at $0$. Then, the tilde-triangle product $a \widetilde{\triangleright}B$ can also be defined when $a \in {{\bf Z}}^{{\mathbb{N}}_+^{op}}$ is a general sequence and the matrix $B\in ({\mathbf{Z}^{{\mathbb{N}^{op}_+}}})^{{\mathbb{N}_+}}$ satisfies that $\text{bin}^{-1} \cdot B\in (({\mathbf{Z}^{{\mathbb{N}_+}}})_{\text{fin}})^{{\mathbb{N}^{op}_+}}.$
Indeed, in this case we can define
$$a \widetilde{\triangleright}B=a \cdot (\text{bin}^{-1} \cdot B)$$
We point out that this definition is compatible with the first one when $a \in ({{\bf Z}}^{{\mathbb{N}}_+^{op}})_{\text{fin}} $ since the matricial product is associative whenever it has sense.
\end{obs}

\subsection{The sequential cardinal functor}

Now we recall from \cite{GHR2022} the relationship between the category of augmented semi-simplicial finite sets and the category of augmented integer sequences. The key point is the sequential cardinal functor, which applies every finite augmented semi-simplicial set to the sequence constituted by the cardinal of the set of $n$-simplices. This sequential cardinal functor also preserves certain structures.

Given any functor $X:\mathbf{\Gamma}_+^{op} \to \mathbf{Sets}_{\text{fin}}$ we may consider the diagram
$$
\xymatrix{\ar[d]_{\text{in}} \mathbb{N}_+^{op} &  {\mathbf Z}\\
{\mathbf{\Gamma}_+^{op}} \ar[r]_X  & {\mathbf{Sets}_{\text{fin}}} \ar[u]_{|-|}}
$$
\noindent where $\text{in}$ denotes the inclusion functor and $|-|$ the functor that gives the cardinal of any finite set.
\medskip
\begin{defi} The \emph{sequential cardinal} of a $\mathbf{\Gamma}_+^{op}$- finite set, $X\in (\mathbf{Sets}_{\text{fin}})^{\mathbf{\Gamma}_+^{op}},$ is defined as the augmented sequence:
$$|X|:\mathbb{N}_+^{op}  \to {\mathbf Z}$$ \noindent given by the composite $|X|:=|-|\circ X\circ \text{in}$, that is,
$$|X|_n:=|X_n|,\hspace{8pt}n\in {\mathbb{N}_+}.$$
\end{defi}

We observe that there is an induced functor $|-|:(\mathbf{Sets}_{\text{fin}})^{\mathbf{\Gamma}_+^{op}} \to {\mathbf Z}^{\mathbb{N}_+^{op}}$ where, for morphisms $f:X \to Y$, it is defined as $|f|:=|Y|-|X|$; that is, $|f|_n=|Y|_n-|X|_n, n \in \mathbb{N}_+$.

On the one hand, we showed in \cite{GHR2022} that $$((\mathbf{Sets}_{\text{fin}})^{\mathbf{\Gamma}_+^{op}}, \sqcup, \Gamma_+^\emptyset)\hspace{10pt}\mbox{and} \hspace{10pt} ((\mathbf{Sets}_{\text{fin}})^{\mathbf{\Gamma}_+^{op}}, \times, \Gamma_+^1)$$  are monoidal categories.  On the other, the ring structure $({\mathbf Z},+, \times)$ induces monoidal structures $( {\mathbf Z}^{\mathbb{N}_+^{op}}, +,0)$, $( {\mathbf Z}^{\mathbb{N}_+^{op}}, \times ,1)$.  Taking into account the identities:
$$|X\sqcup Y|=|X|+|Y|, \quad |X\times Y|=|X|{\times} |Y|$$
$$|\Gamma_+^\emptyset|=0, \quad |\Gamma_+^1|=1$$
\noindent we have that the functor $|-|:(\mathbf{Sets}_{\text{fin}})^{\mathbf{\Gamma}_+^{op}} \to {\mathbf Z}^{\mathbb{N}_+^{op}}$ preserves the monoidal structures induced by coproducts and products (see \cite{GHR2022}):
$$|-|:((\mathbf{Sets}_{\text{fin}})^{\mathbf{\Gamma}_+^{op}}, \sqcup, \Gamma_+^\emptyset)  \to( {\mathbf Z}^{\mathbb{N}_+^{op}}, +,0)$$
$$|-|:((\mathbf{Sets}_{\text{fin}})^{\mathbf{\Gamma}_+^{op}}, \times, \Gamma_+^1) \to( {\mathbf Z}^{\mathbb{N}_+^{op}}, \times ,1).$$

\bigskip
\begin{defi}\label{complejosdiscoyesferaysubdivisones}
Associated to the $\mathbf{\Gamma}_+^{op}$-sets, $\Gamma_+[n]$ and
$S_+[n-1]=\Gamma_+[n]\setminus \{\iota_n\}$, where $\iota_n$ is the identity of $[n] \in \mathbf{\Gamma}_+$,
we consider the sequential cardinals:
$${\gamma_+}[n]:=|{\Gamma_+}[n]|, \quad s_+[n-1]:=|S_+[n-1]|.$$
\medskip
For every $n \in \mathbb{N}_+$, the sequential cardinal ${\gamma_+}[n]=|{\Gamma_+}[n]|$ is given by the binomial coefficients:
$$\begin{array}{lll}
|{\Gamma_+}[n]| & = & (|{\Gamma_+}[n]_{-1}|, |{\Gamma_+}[n]_0|,|\Gamma[n]_1|, \cdots, |{\Gamma_+}[[n]_n|,|\emptyset |,|\emptyset |,\cdots) \\
& = & \small{\left (\binom{n+1}{0},\binom{n+1}{1} ,\binom{n+1}{2}, \cdots, \binom{n+1}{n+1} ,0,0, \cdots \right ).}
\end{array}$$
\end{defi}

\bigskip
Now we consider $((\mathbf{Sets}_{\rm{fin}})^{\mathbf{\Gamma}_+^{op}}, \boxplus, {\Gamma_+}[-1])$, which is a  monoidal subcategory of
$((\mathbf{Sets})^{\mathbf{\Gamma}_+^{op}}, \boxplus, {\Gamma_+}[-1])$. The sequential cardinal functor preserves the monoidal structure, that is
$$
|X \boxplus Y|=|X| \boxplus |Y|, \quad | \Gamma_+[-1]|={\bf 1}_{-1}
$$
\noindent for all $X,Y\in (\mathbf{Sets}_{\text{fin}})^{\mathbf{\Gamma}_+^{op}}$. In other words (see \cite{GHR2022}) the sequential cardinal functor $$|\cdot|:((\mathbf{Sets}_{\rm{fin}})^{\mathbf{\Gamma}_+^{op}}, \boxplus, \Gamma_+[-1]) \to ({\mathbf Z}^{\mathbb{N}_+^{op}}, \boxplus, {\bf 1}_{-1})$$ \noindent is monoidal.

\bigskip
The cone functors for semi-simplicial sets and augmented sequences are also related through the sequential cardinal functor:
\medskip
\begin{prop}\label{conocardinal} The following diagrams are commutative:
$$
\xymatrix{\ar[d]_{|\cdot|} (\mathbf{Sets}_{\text{fin}})^{\mathbf{\Gamma}_+^{op}} \ar[r]^{\text{Con}_l} & (\mathbf{Sets}_{\text{fin}})^{\mathbf{\Gamma}_+^{op}}\ar[d]^{|\cdot|}
 & \ar[d]_{|\cdot|} (\mathbf{Sets}_{\text{fin}})^{\mathbf{\Gamma}_+^{op}} \ar[r]^{\text{Con}_r} & (\mathbf{Sets}_{\text{fin}})^{\mathbf{\Gamma}_+^{op}}\ar[d]^{|\cdot|} \\
\mathbf {Z}^{{\mathbb{N}_+^{op}}}\ar[r]_{\text{con}} & \mathbf {Z}^{{\mathbb{N}_+^{op}}} &
\mathbf {Z}^{{\mathbb{N}_+^{op}}}\ar[r]_{\text{con}} & \mathbf {Z}^{{\mathbb{N}_+^{op}}}}
$$
Moreover, as $\gamma_+[0]=|{\Gamma_+}[0]|={\bf{c}}$, we have $\gamma_+[n]=|{\Gamma_+}[n]|=\boxplus^{n+1}{\mathbf{c}}$ (the $(n+1)$-fold join of $\mathbf{c}$ with itself).
\end{prop}

The cardinal functor $| \cdot |  \colon \mathbf{Sets_{\text{fin}}}\to \mathbf {Z}$ induces a canonical funtor
$|\cdot|:((\mathbf{Sets}_{\text{fin}})^{\mathbf{\Gamma}_+^{op}})^{\mathbf{\Gamma}_+} \to ({\mathbf Z}^{\mathbb{N}_+^{op}})^{\mathbb{N}_+},$ $Z \mapsto |Z|, $
where $|Z|= |\cdot | \circ Z \circ \text{in}$ is the composite
$$
\xymatrix{ \mathbb{N}_+ \ar[r]^{\text{in}} & \Gamma_+  \ar[r]^(.3){Z} & (\mathbf{Sets}_{\text{fin}})^{\mathbf{\Gamma}_+^{op}}
\ar[r]^(.6){|\cdot |} & {\mathbf {Z}^{\mathbb{N}_+^{op}}}.}
$$
In particular, the \emph{augmented Pascal matrix} is defined as the augmented co-sequence given by the composite:
$$
\xymatrix{ \mathbb{N}_+ \ar[r]^{\text{in}} & \Gamma_+  \ar[r]^(.35){\Yoneda} & (\mathbf{Sets}_{\text{fin}})^{\mathbf{\Gamma}_+^{op}}
\ar[r]^(.6){|\cdot |} & {\mathbf {Z}^{\mathbb{N}_+^{op}}}.}
$$
Note that  $ |\cdot | \circ \Yoneda \circ \text{in}=\text{bin}$. As a consequence, see \cite{GHR2022}, the functor
$$|\cdot|:((\mathbf{Sets}_{\text{fin}})^{\mathbf{\Gamma}_+^{op}})^{\mathbf{\Gamma}_+} \to ({\mathbf Z}^{\mathbb{N}_+^{op}})^{\mathbb{N}_+} $$
\noindent carries the Yoneda augmented semi-cosimplicial set $\Yoneda:\mathbf{{\Gamma}_+} \to \mathbf{Sets}^{\mathbf{\Gamma}_+^{op}}$, matricially represented as
$$\Yoneda=\begin{pmatrix}
{\Gamma_+}[-1]\\
{\Gamma_+}[0]\\
\vdots\\
\end{pmatrix}
$$
to the augmentend Pascal matrix, where each row is the cone of the previous one.
$$|\Yoneda|=\begin{pmatrix}
|{\Gamma_+}[-1]|\\
|{\Gamma_+}[0]|\\
\vdots\\
\end{pmatrix}
=
\begin{pmatrix}
{\gamma_+}[-1]\\
{\gamma_+}[0]\\
\vdots\\
\end{pmatrix}
=
\text{bin}
$$

To finish this section, we establish an important result which asserts that, under mild restrictions, the sequential cardinal functor carries the action to the triangle product. We recall from \cite{GHR2022} that a functor  $Z:\mathbf{\Gamma}_+ \to \mathbf{Sets}^{\mathbf{\Gamma}_+^{op}}$ is said to be \emph{regular} if $Z(\varphi)$ is injective
(on each dimension) for every morphism $\varphi $ in $\mathbf{\Gamma}_+ $.

\medskip
\begin{teo}\label{tildebigodottildedot}
If $(\mathbf{Sets}^{\mathbf{\Gamma}_+^{op}})_{\text{reg}}^{\mathbf{\Gamma}_+}$ stands for the full subcategory of
$(\mathbf{Sets}^{\mathbf{\Gamma}_+^{op}})^{\mathbf{\Gamma}_+}$
consisting of regular functors, then the following diagram is commutative:
$$
\xymatrix{\ar[d]_{|-| \times |-|} (\mathbf{Sets}^{\mathbf{\Gamma}_+^{op}})_{\text{fin}}  \times ((\mathbf{Sets}_{\text{fin}})^{\mathbf{\Gamma}_+^{op}})_{\text{reg}}^{\mathbf{\Gamma}_+}
\ar[rr]^(.65){(-)\widetilde{\vartriangleright}(-)} & &  (\mathbf{Sets}_{\text{fin}})^{\mathbf{\Gamma}_+^{op}} \ar[d]^{|-|} \\
({{\mathbf Z}}^{{\mathbb{N}}_+^{op}})_{\text{fin}} \times ({{\mathbf Z}}^{{\mathbb{N}}_+^{op}})^{\mathbb{N}_+} \ar[rr]_(.6){(-)\widetilde{\triangleright}(-)} & & {{\mathbf Z}}^{{\mathbb{N}}_+^{op}}
}
$$
In other words, if $X \in (\mathbf{Sets}^{\mathbf{\Gamma}_+^{op}} )_{\text{fin}}$ and $Z \in ((\mathbf{Sets}_{\text{fin}})^{\mathbf{\Gamma}_+^{op}})_{\text{reg}}$, then
$$|X \tilde{\vartriangleright} Z|=|X| \tilde{\triangleright} |Z|.$$
Moreover, if we specialize $Z:={\Gamma_+}[-]$, then
$|X \widetilde{\vartriangleright} {\Gamma_+}[-]|=|X| \tilde{\triangleright} |{\Gamma_+}[-]|=|X|.$
\end{teo}

\begin{obs}  In the diagram above, instead of the categories  $(\mathbf{Sets}^{\mathbf{\Gamma}_+^{op}})_{\text{fin}}$ and $({{\mathbf Z}}^{{\mathbb{N}}_+^{op}})_{\text{fin}}$, we can take the larger categories $(\mathbf{Sets}_{\text{fin}})^{\mathbf{\Gamma}_+^{op}}$ and ${{\mathbf Z}}^{{\mathbb{N}}_+^{op}}$ if we reduce $((\mathbf{Sets}_{\text{fin}})^{\mathbf{\Gamma}_+^{op}})_{\text{reg}}^{\mathbf{\Gamma}_+}$ to co-simplicial objects $Z$ verifying that for each $q \in \mathbb{N}_+$,  there is $n_q  \in \mathbb{N}_+$ such that $|\breve Z([n])_q|=0$ for every $n\geq n_q$; here the set $\breve Z([n])_q$ is given by
$$\begin{array}{ll}
\breve{Z}([-1])_q:=Z([-1])_q & \\
\breve{Z}([n])_q:=Z([n])_q\setminus (\cup_{i=0}^n(\varphi_i)_*(Z([n-1])_q), & n\geq 0
\end{array}$$
where  $\varphi_i \colon [n-1] \to [n]$ is the canonical $i$-th increasing inclusion $i\in \{0, \cdots, n\}$.
\end{obs}

\section{Cylinders and barycentric subdivisions}

\subsection{Cylinders of an augmented semi-simplicial set}\label{geometriccylinders}

\subsubsection{The standard cylinder}


Given $\sigma \in {\Gamma_+}[n]_p$ and $\tau \in {\Gamma_+}[n]_q$, we will write $\sigma \preccurlyeq\tau $ whenever the following two conditions hold:
\begin{enumerate}
\item[(i)] $|\sigma([p]) \cap \tau([q])| \leq 1 $

\item[(ii)] $\sigma(i) \leq \tau (j)$, for all $ i\in  [p]$ and $j\in [q]$
\end{enumerate}
Observe that if $\sigma \in {\Gamma_+}[n]_{-1}$ or $\tau \in {\Gamma_+}[n]_{-1}$, then $\sigma \preccurlyeq\tau $.

For each object $[n] \in {\mathbf{\Gamma}_+}$ we define the \emph{(standard) cylinder of $\Gamma_+[n]$}, denoted by $\text{Cil}{\Gamma_+}[n]$, as the following augmented semi-simplicial set. If $m\in \mathbf{\Gamma_+}$ we consider two possibilities: $m\geq 0$ and $m=-1.$ Now, if $m\geq 0$ we take
$$\text{Cil}{\Gamma_+}[n]_m:=\{ (\sigma, \tau)\in {\Gamma_+}[n]_p\times {\Gamma_+}[n]_q\hspace{2pt}:\hspace{2pt}p,q \geq 0,\hspace{2pt}p+q=m-1\hspace{2pt} \mbox{and}\hspace{2pt} \sigma  \preccurlyeq \tau \}$$

On the other hand, if $m=-1$ we take $\text{Cil}{\Gamma_+}[n]_{-1}=\{(\emptyset ,\emptyset)\}.$
Here, $\emptyset$ denotes the unique element $\emptyset=[-1] \to [n]$ belonging to $\Gamma_+[n]_{-1}$.
\begin{obs}
With this definition observe that $\text{Cil}{\Gamma_+}[-1]_{-1}=\{(\emptyset,\emptyset)\}=\{*\}$ and $\text{Cil}{\Gamma_+}[-1]_{m}=\emptyset $, for all $m\geq 0.$
\end{obs}
Now, for any $m\geq 0$, $(\sigma ,\tau )\in  \text{Cil}{\Gamma_+}[n]_m$ and $0\leq i\leq m$, the face operator $d_i:\text{Cil}{\Gamma_+}[n]_m\rightarrow \text{Cil}{\Gamma_+}[n]_{m-1}$ is given by
$$
d_i((\sigma,\tau)) =
\left\{
\begin{array}{lll}
 (d_i(\sigma), \tau) & \mbox{if $0\leq i \leq p $},\\ [1pc]
(\sigma, d_{i-p-1}(\tau))  & \mbox{ if $p+1 \leq i \leq m$}.
\end{array} \right.
$$
It easy to check that we have an induced augmented semi-simplicial set $$\text{Cil}{\Gamma_+}[n] \colon {\mathbf{\Gamma_+^{op}}} \to \Set$$

\textbf{Notation:}
For the sake of conciseness and clarity we will consider the following notation. For $m\geq 0$, since any strictly increasing map $\sigma \colon [m]\to [n]$ is given by $0\leq \sigma(0)<\sigma(1)< \cdots < \sigma(m) \leq n$, then we may write $\sigma $ as the tuple $\sigma=(\sigma(0),\sigma(1), \cdots, \sigma(m))$.
As far as a pair $(\sigma, \tau)$ is concerned, it will be denoted according to the following possibilities:
\begin{itemize}
\item If $\sigma=(\sigma(0),\sigma(1), \cdots, \sigma(p))$ and $\tau=(\tau(0),\tau(1), \cdots, \tau(q))$, then we will write the pair $(\sigma, \tau)$ as
$(\sigma(0),\sigma(1), \cdots, \sigma(p),\tau(0)\sp{\prime},\tau(1)\sp{\prime}, \cdots, \tau(q)\sp{\prime}).$

\item If $\sigma\not = \emptyset$ and $\tau=\emptyset$, then $(\sigma , \emptyset)$ is reduced to $(\sigma(0),\sigma(1), \cdots, \sigma(p))$. Similarly,
if $\sigma= \emptyset$ and $\tau\not =\emptyset$, then $(\emptyset,\tau )$ is reduced to $(\tau(0)\sp{\prime},\tau(1)\sp{\prime}, \cdots, \tau(q)\sp{\prime})$.

\item Finally, if $\sigma= \emptyset$ and $\tau =\emptyset$, we use the notation $(\emptyset,\emptyset)=*$
\end{itemize}

\begin{ej}
Using the notation described above, we always have $${\rm{Cil}}{\Gamma_+}[n]_{-1}=\{*\}(=\{(\emptyset, \emptyset)\})$$ \noindent for all $n$.
Next we describe ${\rm{Cil}}{\Gamma_+}[0]$, ${\rm{Cil}}{\Gamma_+}[1]$ and ${\rm{Cil}}{\Gamma_+}[2]$:
\begin{enumerate}
\item
For $n=0$ we have
${\rm{Cil}}{\Gamma_+}[0]_{-1}=\{*\}$ and ${\rm{Cil}}{\Gamma_+}[0]_{0}=\{((0), \emptyset), ( \emptyset, (0))\}=\{0, 0\sp{\prime}\}$.
Now, if $m=1$, then $p$ and $q$ must satisfy $p+q=0$ and we have $(p,q)\in \{(-1,1), (1,-1), (0,0)\}$. Since  ${\Gamma_+}[0]_{1}$ is the empty set we obtain
${\rm{Cil}}{\Gamma_+}[0]_{1}=\{((0), (0))\}=\{(0, 0\sp{\prime})\}$. Moreover, ${\rm{Cil}}{\Gamma_+}[0]_{m}=\emptyset ,$ for all $m\geq 2.$

\item For $n=1$ we have

${\rm{Cil}}{\Gamma_+}[1]_{-1}=\{*\}$

${\rm{Cil}}{\Gamma_+}[1]_{0}=\{((0), \emptyset) , ((1), \emptyset), (\emptyset, (0)),(\emptyset, (1)) \}=\{0,1, 0\sp{\prime},1\sp{\prime}\}$

${\rm{Cil}}{\Gamma_+}[1]_{1}=\{( (0,1), \emptyset ),(\emptyset, (0,1)), ((0), (1)), ((0), (0)), ((1), (1)) \}\\=\{(0,1), (0\sp{\prime},1\sp{\prime}), (0,1\sp{\prime}), (0,0\sp{\prime}), (1,1\sp{\prime})\}$

If $m=2$, then $p$ and $q$ must satisfy $p+q=1$ so $(p,q)\in \{(-1,2), (0,1), (1,0)\}$. Since ${\Gamma_+}[1]_{2}$ is the empty set we obtain

 ${\rm{Cil}}{\Gamma_+}[1]_{2}=\{(0,(0,1)), ((0,1), 1)\}=\{(0,0\sp{\prime},1\sp{\prime}), (0,1,1\sp{\prime})\}$

If $m=3$, then $p$ and $q$ must satisfy $p+q=2$ so $(p,q)\in \{(-1,3), (0,2), (1,1)\}$. Note that ${\Gamma_+}[1]_{3}$, ${\Gamma_+}[1]_{2}$ are empty sets. For $(p,q)=(1,1)$, we have $\sigma=(0,1)=\tau$ and $\sigma \not
\preccurlyeq\tau$ in this case. This implies that  ${\rm{Cil}}{\Gamma_+}[1]_{3}= \emptyset$; moreover, ${\rm{Cil}}{\Gamma_+}[1]_{m}= \emptyset$, for all $m\geq 3.$

\item For $n=2$ (see Figure \ref{cilindro2}) we have the following sets:
\footnotesize{

${\rm{Cil}}{\Gamma_+}[2]_{-1}=\{*\}$

${\rm{Cil}}{\Gamma_+}[2]_{0}=\{0,1, 2, 0\sp{\prime},1\sp{\prime}, 2\sp{\prime}\}$

${\rm{Cil}}{\Gamma_+}[2]_{1}=\{(0,1),(0,2), (0,0\sp{\prime}),(0,1\sp{\prime}), (0,2\sp{\prime}), (1,2), (1,1\sp{\prime}),  (1,2\sp{\prime}),  (2,2\sp{\prime}), (0\sp{\prime},1\sp{\prime}), (0\sp{\prime},2\sp{\prime}),(1\sp{\prime},2\sp{\prime}) \}$

${\rm{Cil}}{\Gamma_+}[2]_{2}=\{(0,1,2),  (0,1,1\sp{\prime}), (0,1,2\sp{\prime}), (0,2,2\sp{\prime}),  (0,0\sp{\prime},1\sp{\prime}), (0,0\sp{\prime},2), (0,1\sp{\prime},2\sp{\prime}), (1,2,2\sp{\prime}),(1,1\sp{\prime},2\sp{\prime}), (0\sp{\prime},1\sp{\prime},2\sp{\prime})  \}$

${\rm{Cil}}{\Gamma_+}[2]_{3}=\{(0,1,2,2\sp{\prime}), (0,1,1\sp{\prime},2\sp{\prime}),  (0,0\sp{\prime},1\sp{\prime},2\sp{\prime})\}$

${\rm{Cil}}{\Gamma_+}[2]_{m}=\emptyset$, $m \geq 4$.}
\end{enumerate}
\end{ej}

\begin{figure}[htbp]
\begin{center}
\includegraphics[scale=0.7]{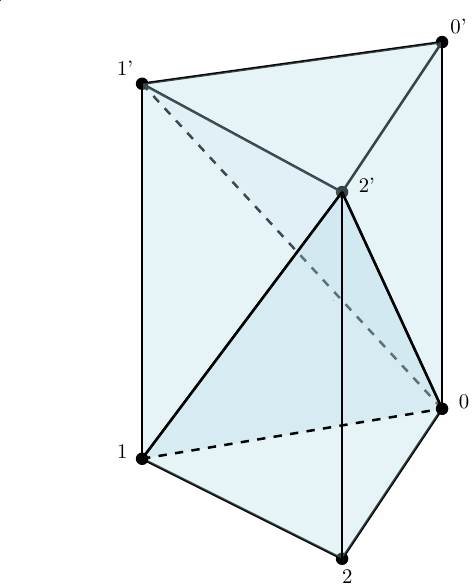}
\caption{The cylinder of $\Gamma_+[2]$}
\label{cilindro2}
\end{center}
\end{figure}

\bigskip
One can straightforwardly check that we have an induced functor $\text{Cil} \colon {\mathbf{\Gamma_+}} \to \Set ^{\mathbf{\Gamma}_+^{op}}, \text{Cil} ([n])= \text{Cil}{\Gamma_+}[n]$.  By Theorem \ref{ext-pr} it is obtained
the colimit-preserving functor $(-)\widetilde{\vartriangleright} \text{Cil} :\mathbf{Sets}^{\mathbf{\Gamma}_+^{op}}\to \mathbf{Sets}^{\mathbf{\Gamma}_+^{op}}$ making commutative the diagram
$$
\xymatrix{ \ar@{^(->}[d]_{\Yoneda} {\mathbf{\Gamma_+}} \ar[r]^{\text{Cil}} & \mathbf{Sets}^{\mathbf{\Gamma}_+^{op}}\\
\mathbf{Sets}^{\mathbf{\Gamma}_+^{op}}\ar[ru]_{(-)\widetilde{\vartriangleright} \text{Cil} } & }, $$

\begin{defi}
For any augmented semi-simplicial set $X \in \mathbf{Sets}^{\mathbf{\Gamma}_+^{op}}$ we define its \emph{(standard) cylinder} as the augmented semi-simplicial set
$$\text{Cil}(X):= X \widetilde{\vartriangleright} \text{Cil}$$ Similarly for augmented semi-simplicial maps, $\text{Cil}(f):= f \widetilde{\vartriangleright} \text{Cil}$.
\end{defi}

The functor $\text{Cil}$ has a right adjoint
$$\text{r}_{\text{Cil}} \colon \Set^{{\mathbf{\Gamma}_+^{op}}} \to \Set^{{\mathbf{\Gamma}_+^{op}}}.$$

If we have an augmented simplicial subset $X \subset \Gamma_+[n]$, then we can take $\text{Cil}X$ as
$\text{Cil}X_m=\{(\sigma, \tau) \in \text{Cil}{\Gamma_+}[n]_m\hspace{2pt}|\hspace{2pt}\mbox{there is}\hspace{2pt}\gamma \in X_m\hspace{2pt}\text{with }\hspace{2pt}(\mbox{Im}(\sigma )\cup \mbox{Im}(\tau )) \subset \mbox{Im}(\gamma )\}.$

Therefore, we can define inductively the cylinder $\text{Cil}{\Gamma_+}[n]$ as follows. Suppose that  $\sigma \colon  [m] \to [n]$ is given by $\sigma(0)<\sigma(1)< \cdots < \sigma(m)$. Then:

\begin{itemize}
\item For each $0\leq p\leq m$, we have the  $(m+1)$-simplex

$(\sigma(0),0)<(\sigma(1),0)< \cdots < (\sigma(p),0)< (\sigma(p),1)< \cdots <(\sigma(m),1)=\\(\sigma(0), \cdots , \sigma(p), \sigma(p)',\cdots ,\sigma(m)')$

\item and for each $-1 \leq q \leq m+1$ we have the $m$-simplex

$(\sigma(0),0)<(\sigma(1),0)< \cdots < (\sigma(q),0)< (\sigma(q+1),1)< \cdots <(\sigma(m),1)=\\(\sigma(0), \cdots , \sigma(q), \sigma(q+1)',\cdots ,\sigma(m)')$, where

for $q=-1$,
$(\sigma(0),1)<(\sigma(1),1)< \cdots <(\sigma(m),1)=(\sigma'(0) ,\cdots ,\sigma(m)')$,

and for $q=m+1$,
$(\sigma(0),0)<(\sigma(1),0)< \cdots <(\sigma(m),0)=(\sigma(0) ,\cdots ,\sigma(m))$.

\end{itemize}

\begin{obs} When studying simplicial theory, the standard way to triangulate the product prism $\Delta[p] \times\Delta[1]$ is by taking the $(p + 1$)-simplices $(0, \cdots , k, k', \cdots , p')$, where the numbers without the prime symbol represent vertices in $\Delta[p] \times \{0\}$ and the  numbers with the prime symbol represent vertices in $\Delta[p] \times \{1\}$. The simplex $(0,\cdots,k,k',\cdots,p')$ corresponds to $k+1$ zeros and $p-k+1$ ones.
\end{obs}

\bigskip
Take $\partial \Gamma_+[n]$, the boundary of $\Gamma_+[n]$, and $\breve{\text{Cil}}{\Gamma_+}[n]:={\text{Cil}}{\Gamma_+}[n] \setminus \text{Cil}\partial \Gamma_+[n]$. One can find a pretty straightforward pattern for the cardinal sequences $|\breve{\text{Cil}}{\Gamma_+}[n]|$:
$$
\begin{array}{l |l ccccccccc}
 &-1& 0 & 1 & 2 & 3 & 4 & 5 & 6 & \dots \\\hline
|\breve{\text{Cil}}{\Gamma_+}[-1]|&1& 0 & 0 & 0 & 0 & 0 & 0 & 0 & \dots&  \\
|\breve{\text{Cil}}{\Gamma_+}[0]|&0& 2 & 1 & 0 & 0 & 0 & 0 & 0 &\dots& \\
|\breve{\text{Cil}}{\Gamma_+}[1]|&0& 0 & 3 & 2 & 0 & 0 & 0 & 0 &\dots&  \\
|\breve{\text{Cil}}{\Gamma_+}[2]|&0& 0 & 0 &4  & 3 & 0 &  0& 0 &\dots&  \\
|\breve{\text{Cil}}{\Gamma_+}[3]|&0& 0 &  0& 0 & 5 &  4& 0 &0  & \dots& \\
|\breve{\text{Cil}}{\Gamma_+}[4]|&0& 0 & 0 &0 &0 & 6 & 5 & 0 & \dots& \\
|\breve{\text{Cil}}{\Gamma_+}[5]|&0&  0& 0 &  0& 0 &0 & 7&  6&\dots&  \\
|\breve{\text{Cil}}{\Gamma_+}[6]|&0& 0 & 0 & 0 &  0& 0 & 0 &  8& \dots& \\
~~~~~~~~~~\vdots &\vdots & \vdots  &\vdots  & \vdots & \vdots  &\vdots  & \vdots  & \vdots  & \ddots  &  \\

\end{array}
$$
This permits us the computation of the cardinal sequences $|\text{Cil}\partial{\Gamma_+}[n]|$ and $|\text{Cil}{\Gamma_+}[n]|$.

\begin{ej} For instance, for $n=2$ we have (see proposition below):
$$|{\rm{Cil}}\partial{\Gamma_+}[2]|= |\breve{{\rm{Cil}}}{\Gamma_+}[-1]|+ \binom{2+1}{0+1} |\breve{{\rm{Cil}}}{\Gamma_+}[0]|+ \binom{2+1}{1+1} |\breve{{\rm{Cil}}}{\Gamma_+}[1]|.$$
As
$|\breve{{\rm{Cil}}}{\Gamma_+}[-1]|=(1, 0 , 0 , 0 , 0 , 0 , 0 , 0 , \cdots)$, $|\breve{{\rm{Cil}}}{\Gamma_+}[0]|=(0, 2 , 1 , 0 , 0 , 0 , 0 , 0 ,\cdots)$  and $|\breve{{\rm{Cil}}}{\Gamma_+}[1]|=(0, 0 , 3 , 2 , 0 , 0 , 0 , 0 ,\cdots),$ then
$$|{\rm{Cil}}\partial{\Gamma_+}[2]|=(1, 6 , 12 , 6 , 0 , 0 , 0 , 0 , \cdots).$$
Moreover,
$$\begin{array}{rl}
|{\rm{Cil}}{\Gamma_+}[2]|= & |\breve{{\rm{Cil}}}{\Gamma_+}[-1]|+ \binom{2+1}{0+1} |\breve{{\rm{Cil}}}{\Gamma_+}[0]|+ \binom{2+1}{1+1} |\breve{{\rm{Cil}}}{\Gamma_+}[1]|+\binom{2+1}{2+1} |\breve{{\rm{Cil}}}{\Gamma_+}[2]| \\
=&|{\rm{Cil}}\partial{\Gamma_+}[2]|+|\breve{{\rm{Cil}}}{\Gamma_+}[2]|.
\end{array}$$
As $|{\rm{Cil}}\partial{\Gamma_+}[2]|=(1, 6 , 12 , 6 , 0 , 0 , 0 , 0 , \cdots)$ and  $|\breve{{\rm{Cil}}}{\Gamma_+}[2]|=(0,0,0,4,3,0, \cdots)$, then
$$|{\rm{Cil}}{\Gamma_+}[2]|=(1, 6 , 12 , 10 , 3 , 0 , 0 , 0 , \cdots).$$
\end{ej}

The simple algorithm above has a proper generalization:
\begin{prop}\label{sucesioncardinalCil} For every $n \in \mathbb{N}_+$, we have:

\begin{itemize}

\item[(i)] 
$$
|\breve{\text{Cil}}{\Gamma_+}[n]|_k=
\left\{
\begin{array}{lll}
n+2, & \mbox{if $k=n $},\\ [1pc]
n+1, & \mbox{ if $k=n+1$}\\ [1pc]
0, & \mbox{ if $k\not \in \{n, n+1\}$}
\end{array} \right.
$$

\item[(ii)] $|\text{Cil}\partial{\Gamma_+}[n]|=\sum _{i=-1}^{n-1}\binom{n+1}{i+1} |\breve{\text{Cil}}{\Gamma_+}[i]|,$

\item[(iii)]  $|\text{Cil}{\Gamma_+}[n]|=|\text{Cil}\partial{\Gamma_+}[n]|+|\breve{\text{Cil}}{\Gamma_+}[n]|.$

\end{itemize}
\end{prop}

\bigskip
Now, using the result above, we are able to compute the cardinal sequences $|\text{Cil}\partial{\Gamma_+}[n]|$ and
$|\text{Cil}{\Gamma_+}[n]|$. Here, we display the following tables:

$$
\begin{array}{l |l ccccccccc}
 &-1& 0 & 1 & 2 & 3 & 4 & 5 & 6 & \dots \\\hline
|\text{Cil}\partial{\Gamma_+}[-1]|&0& 0 & 0 & 0 & 0 & 0 & 0 & 0 & \dots&  \\
|\text{Cil}\partial{\Gamma_+}[0]|&1& 0 & 0 & 0 & 0 & 0 & 0 & 0 &\dots& \\
|\text{Cil}\partial{\Gamma_+}[1]|&1& 4 & 2 & 0 & 0 & 0 & 0 & 0 &\dots&  \\
|\text{Cil}\partial{\Gamma_+}[2]|&1& 6 & 12 &6  & 0 & 0 &  0& 0 &\dots&  \\
|\text{Cil}\partial{\Gamma_+}[3]|&1& 8 &22  & 28 & 12 & 0 & 0 & 0& \dots& \\
|\text{Cil}\partial{\Gamma_+}[4]|&1& 10 & 35 & 60 & 55 & 20 &  0& 0 & \dots& \\
|\text{Cil}\partial{\Gamma_+}[5]|&1& 12 &51  & 110 & 135 & 96&30 &  0&\dots&  \\
|\text{Cil}\partial{\Gamma_+}[6]|&1& 14 & 70 & 182 & 280 & 266 &154  & 42 & \dots& \\
~~~~~~~~~~\vdots &\vdots & \vdots  &\vdots  & \vdots & \vdots  &\vdots  & \vdots  & \vdots  & \ddots  &\\

\end{array}
$$

$$
\begin{array}{l |l cccccccccc}
 &-1& 0 & 1 & 2 & 3 & 4 & 5 & 6 &0 &\dots \\\hline
|\text{Cil}{\Gamma_+}[-1]|&1& 0 & 0 & 0 & 0 & 0 & 0 & 0 &0 &\ \dots&  \\
|\text{Cil}{\Gamma_+}[0]|&1& 2 & 1 & 0 & 0 & 0 & 0 & 0 &0 &\dots& \\
|\text{Cil}{\Gamma_+}[1]|&1& 4 & 5 & 2 & 0 & 0 & 0 & 0 &0 &\dots&  \\
|\text{Cil}{\Gamma_+}[2]|&1& 6 & 12 &10  & 3 & 0 &  0& 0 &0 &\dots& \\
|\text{Cil}{\Gamma_+}[3]|&1& 8 &22  & 28 & 17 & 4 & 0 & 0 &0 & \dots& \\
|\text{Cil}{\Gamma_+}[4]|&1& 10 & 35 &60 & 55& 26 & 5 & 0 &0 & \dots& \\
|\text{Cil}{\Gamma_+}[5]|&1&12  & 51 &110  &135  &96 & 37& 6 &0 &\dots &  \\
|\text{Cil}{\Gamma_+}[6]|&1& 14  &70  &182  & 200 &  266&  154 &  50&7 & \dots & \\
~~~~~~~~~~\vdots &\vdots & \vdots  &\vdots  & \vdots & \vdots  &\vdots  & \vdots  & \vdots  &\vdots  & \ddots   &\\

\end{array}
$$

\begin{obs}  For $n\geq1$,  $a(n)=|\text{Cil}{\Gamma_+}[n-1]|_1$ is the sequence of pentagonal numbers  A000326, where $a(n) = \frac{n(3n-1)}{2}$ (\url{https://oeis.org/A000326}).

Now, taking $a(n)=|\text{Cil}{\Gamma_+}[n]|_2$ one has the sequence A006331, where $a(n) = n (n+1)\frac{(2n+1)}{3}$		 (\url{https://oeis.org/A006331}).

We also have that $a(n)=|\text{Cil}{\Gamma_+}[n]|_3$ is the sequence A212415, where $a(n) = (n-1)n(n+1)\frac{(5n+2)}{24 }$		 (\url{https://oeis.org/A212415}).
\end{obs}

\subsubsection{ The 0-cylinder}

Similarly to the case of the standard cylinder, given $\sigma \in {\Gamma_+}[n]_p$ and $\tau \in {\Gamma_+}[n]_q$ we will write $\sigma \prec\tau $
whenever the following two conditions hold:
\begin{enumerate}
\item[(i)] $|\sigma([p]) \cap \tau([q])|=0$

\item[(ii)] $\sigma(i)<\tau (j)$, for all $ i\in  [p]$ and $j\in [q]$
\end{enumerate}

Note that, if $\sigma \in {\Gamma_+}[n]_{-1}$ or $\tau \in {\Gamma_+}[n]_{-1}$, then $\sigma \prec\tau $.

For each object  $[n] \in {\mathbf{\Gamma}_+}$, we introduce $\text{Cil}_0{\Gamma_+}[n]$, the augmented semi-simplicial subset of $\text{Cil}{\Gamma_+}[n]$ defined as:
$$\text{Cil}_0{\Gamma_+}[n]_m= \{ (\sigma, \tau)\in \text{Cil}{\Gamma_+}[n]_m\hspace{2pt}:\hspace{2pt}\sigma \prec \tau \}\subset \text{Cil}{\Gamma_+}[n]_m$$
\noindent for $m\geq 0$ and $\text{Cil}_0{\Gamma_+}[n]_{-1}=\{(\emptyset ,\emptyset)\}=\{*\}.$

We consider the same notation as the one introduced for the standard cylinder.

\begin{ej} Note that  ${\rm{Cil}}_0{\Gamma_+}[-1]_{-1}=\{*\}$ and  ${\rm{Cil}}_0{\Gamma_+}[-1]_m=\emptyset $ for $m\geq 0$.
Next, we describe ${\rm{Cil}}_0{\Gamma_+}[0],$ ${\rm{Cil}}_0{\Gamma_+}[1]$ and  ${\rm{Cil}}_0{\Gamma_+}[2]$:
\begin{enumerate}

\item For $n=0$ we have

${\rm{Cil}}_0{\Gamma_+}[0]_{-1}=\{(\emptyset, \emptyset)\}=\{*\}$,

${\rm{Cil}}_0{\Gamma_+}[0]_{0}=\{( \emptyset, (0)), ((0), \emptyset)\}=\{ 0\sp{\prime}, 0\}.$

If $m=1$, then $p$ and $q$ must satisfy $p+q=0$, and therefore $(p,q)\in \{(-1,1), (1,-1), (0,0)\}$. Since ${\Gamma_+}[0]_{1}$ is the empty set we have

${\rm{Cil}}_0{\Gamma_+}[0]_{1}=\emptyset.$

Moreover, ${\rm{Cil}}_0{\Gamma_+}[0]_{m}=\emptyset$, for all $m\geq 1.$

\item For $n=1$ we have

${\rm{Cil}}_0{\Gamma_+}[1]_{-1}=\{*\}$,

${\rm{Cil}}_0{\Gamma_+}[1]_{0}=\{(\emptyset, (0)),(\emptyset, (1)), ((0), \emptyset) , ((1), \emptyset)  \}=\{ 0\sp{\prime},1\sp{\prime}, 0,1\}$,

${\rm{Cil}}_0{\Gamma_+}[1]_{1}=\{(\emptyset, (0,1)), ((0), (1)), ( (0,1), \emptyset ) \}=\{(0\sp{\prime},1\sp{\prime}), (0,1\sp{\prime}), (0,1)\}$.

If $m=2$, then $p$ and $q$ must satisfy $p+q=1$, so $(p,q)\in \{(-1,2), (0,1), (1,0)\}$. In this case, ${\Gamma_+}[1]_{2}$ is the empty set so

${\rm{Cil}}_0{\Gamma_+}[1]_{2}=\{(0,(0,1)), ((0,1), 1)\}=\emptyset.$

Moreover, ${\rm{Cil}}_0{\Gamma_+}[1]_{m}=\emptyset$, for all $m\geq 2.$

\item For $n=2$ (see Figure \ref{cilindro2}) we have the following sets of simplices which do not contain ``vertical'' 1-simplices:

${\rm{Cil}}_0{\Gamma_+}[2]_{-1}=\{*\}$,

${\rm{Cil}}_0{\Gamma_+}[2]_{0}=\{0\sp{\prime},1\sp{\prime}, 2\sp{\prime}, 0,1, 2\}$,

${\rm{Cil}}_0{\Gamma_+}[2]_{1}=\{ (0\sp{\prime},1\sp{\prime}), (0\sp{\prime},2\sp{\prime}),(1\sp{\prime},2\sp{\prime}),(0,1\sp{\prime}), (0,2\sp{\prime}),   (1,2\sp{\prime}),   (0,1),(0,2), (1,2) \}$,

${\rm{Cil}}_0{\Gamma_+}[2]_{2}=\{     (0\sp{\prime},1\sp{\prime},2\sp{\prime}),   (0,1\sp{\prime},2\sp{\prime}),  (0,1,2\sp{\prime}), (0,1,2)  \}$,

${\rm{Cil}}_0{\Gamma_+}[2]_{m}=\emptyset$, $m \geq 3$.
\end{enumerate}
\end{ej}

\begin{figure}[htbp]
\begin{center}
\includegraphics[scale=0.7]{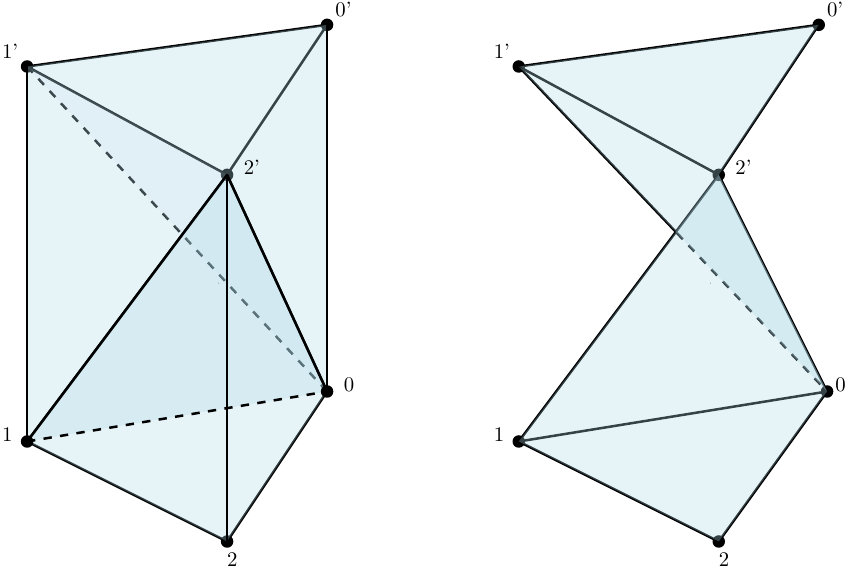}
\caption{The 0-cylinder of $\Gamma_+[2]$ is on the right (it is obtained from the left by removing simplices which contain some vertical 1-simplex)}
\label{cilindro2}
\end{center}
\end{figure}

Again, we have an induced functor $\text{Cil}_0\colon {\mathbf{\Gamma_+}} \to \Set ^{\mathbf{\Gamma}_+^{op}}, \, \text{Cil}_0([n])=\text{Cil}_0\Gamma_+[n]$ and, by Theorem \ref{ext-pr}, one obtains the colimit-preserving functor $(-)\widetilde{\vartriangleright} \text{Cil}_0{\Gamma_+} :\mathbf{Sets}^{\mathbf{\Gamma}_+^{op}}\to \mathbf{Sets}^{\mathbf{\Gamma}_+^{op}}$, making commutative the diagram
$$
\xymatrix{ \ar@{^(->}[d]_{\Yoneda} {\mathbf{\Gamma_+}} \ar[r]^{\text{Cil}_0 } & \mathbf{Sets}^{\mathbf{\Gamma}_+^{op}}\\
\mathbf{Sets}^{\mathbf{\Gamma}_+^{op}}\ar[ru]_{(-)\widetilde{\vartriangleright} \text{Cil}_0} & }, $$

\begin{defi}
For any augmented semi-simplicial set $X \in \mathbf{Sets}^{\mathbf{\Gamma}_+^{op}}$ we define its 0-cylinder as
$$\text{Cil}_0(X):= X \widetilde{\vartriangleright} \text{Cil}_0$$
\noindent and similarly, for augmented semi-simplicial maps $\text{Cil}_0f:= f \widetilde{\vartriangleright} \text{Cil}_0$.
\end{defi}

This functor $\text{Cil}_0$  has a right adjoint $\text{r}_{\text{Cil}_0} \colon \Set^{{\mathbf{\Gamma}_+^{op}}} \to \Set^{{\mathbf{\Gamma}_+^{op}}}.$

Note that, if we have an augmented semi-simplicial subset $X \subset \Gamma_+[n]$, we can take
$\text{Cil}_0X=\{(\sigma, \tau) \in \text{Cil}_0{\Gamma_+}[n]\hspace{2pt}|\hspace{2pt}\text{there is }\hspace{2pt} \gamma \in X_m\hspace{2pt}\text{with}\hspace{2pt}( \text{Im}(\sigma) \cup \text{Im}(\tau))\subset  \text{Im}(\gamma) \}.$

In the same manner as the standard cylinder, denote $\breve{\text{Cil}}_0{\Gamma_+}[n]={\text{Cil}_0}{\Gamma_+}[n] \setminus \text{Cil}_0\partial \Gamma_+[n]$. In this case, we even find a simpler pattern for the cardinal sequences $|\breve{\text{Cil}}_0{\Gamma_+}[n]|$:
$$
\begin{array}{l |l ccccccccc}
 &-1& 0 & 1 & 2 & 3 & 4 & 5 & 6 & \dots \\\hline
|\breve{\text{Cil}}_0{\Gamma_+}[-1]|&1& 0 & 0 & 0 & 0 & 0 & 0 & 0 & \dots&  \\
|\breve{\text{Cil}}_0{\Gamma_+}[0]|&0& 2 & 0 & 0 & 0 & 0 & 0 & 0 &\dots& \\
|\breve{\text{Cil}}_0{\Gamma_+}[1]|&0& 0 & 3 & 0 & 0 & 0 & 0 & 0 &\dots&  \\
|\breve{\text{Cil}}_0{\Gamma_+}[2]|&0& 0 & 0 &4  & 0 & 0 &  0& 0 &\dots&  \\
|\breve{\text{Cil}}_0{\Gamma_+}[3]|&0& 0 &  0& 0 & 5 &  0& 0 &0  & \dots& \\
|\breve{\text{Cil}}_0{\Gamma_+}[4]|&0& 0 & 0 &0 &0 & 6 & 0 & 0 & \dots& \\
|\breve{\text{Cil}}_0{\Gamma_+}[5]|&0&  0& 0 &  0& 0 &0 & 7&  0&\dots&  \\
|\breve{\text{Cil}}_0{\Gamma_+}[6]|&0& 0 & 0 & 0 &  0& 0 & 0 &  8& \dots& \\
~~~~~~~~~~\vdots &\vdots & \vdots  &\vdots  & \vdots & \vdots  &\vdots  & \vdots  & \vdots  & \ddots  &  \\

\end{array}
$$
\noindent which allows us to compute the sequences $|\text{Cil}_0\partial{\Gamma_+}[n]|$ and $|\text{Cil}_0\partial{\Gamma_+}[n]|$.

\begin{prop} \label{Cilcerocal} For every $n \in \mathbb{N}_+$, we have:

\begin{itemize}
\item[(i)] 
$$
|\breve{\text{Cil}}_0{\Gamma_+}[n]|_k=
\left\{
\begin{array}{lll}
n+2, & \mbox{if $k=n $},\\ [1pc]
0, & \mbox{ if $k\not= n$}
\end{array} \right.
$$

\item[(ii)] $|\text{Cil}_0\partial{\Gamma_+}[n]|=\sum _{i=-1}^{n-1}\binom{n+1}{i+1}  |\breve{\text{Cil}}_0{\Gamma_+}[i]|$.

\item[(iii)] $|\text{Cil}_0{\Gamma_+}[n]|=|\text{Cil}_0\partial{\Gamma_+}[n]|+|\breve{\text{Cil}}_0{\Gamma_+}[n]|$.

\end{itemize}
\end{prop}

\bigskip

\begin{ej} For instance, using Proposition \ref{Cilcerocal} parts (i) and (ii), we have for $n=2$ that
$$|{\rm{Cil}}_0\partial{\Gamma_+}[2]|= |\breve{{\rm{Cil}}}_0{\Gamma_+}[-1]|+ 3|\breve{{\rm{Cil}}}_0{\Gamma_+}[0]|+ 3|\breve{{\rm{Cil}}}_0{\Gamma_+}[1]|$$
\noindent and

$|\breve{{\rm{Cil}}}_0{\Gamma_+}[-1]|=(1, 0 , 0 , 0 , 0 , 0 , 0 , 0 , \cdots)$

$|\breve{{\rm{Cil}}}_0{\Gamma_+}[0]|=(0, 2 , 0 , 0 , 0 , 0 , 0 , 0 ,\cdots)$

$|\breve{{\rm{Cil}}}_0{\Gamma_+}[1]|=(0, 0 , 3 , 0, 0 , 0 , 0 , 0 ,\cdots)$

Therefore, $|{\rm{Cil}}_0\partial{\Gamma_+}[2]|=(1, 6 , 9 , 0 , 0 , 0 , 0 , 0 , \cdots)$.

Finally, using again Proposition \ref{Cilcerocal}, and taking into account that the cardinal sequence $|\breve{{\rm{Cil}}}_0{\Gamma_+}[2]|=(0, 0 , 0 , 4, 0 , 0 ,\cdots)$,  we have
$$|{\rm{Cil}}_0{\Gamma_+}[2]|=
|{\rm{Cil}}_0\partial{\Gamma_+}[2]|+|\breve{{\rm{Cil}}}_0{\Gamma_+}[2]|= (1, 6 , 9 , 4 ,0 , 0 , 0 , \cdots).$$
\end{ej}

\bigskip
Using this proposition, we are able to compute the cardinal sequences $|{\rm{Cil}}_0\partial{\Gamma_+}[n]|$ and
$|{\rm{Cil}}_0{\Gamma_+}[n]|$. Here we list the first examples:
$$
\begin{array}{l |l ccccccccc}
 &-1& 0 & 1 & 2 & 3 & 4 & 5 & 6 & \dots \\\hline
|{\rm{Cil}}_0\partial{\Gamma_+}[-1]|&0& 0 & 0 & 0 & 0 & 0 & 0 & 0 & \dots&  \\
|{\rm{Cil}}_0\partial{\Gamma_+}[0]|&1& 0 & 0 & 0 & 0 & 0 & 0 & 0 &\dots& \\
|{\rm{Cil}}_0\partial{\Gamma_+}[1]|&1& 4 & 0 & 0 & 0 & 0 & 0 & 0 &\dots&  \\
|{\rm{Cil}}_0\partial{\Gamma_+}[2]|&1& 6 & 9 &0  & 0 & 0 &  0& 0 &\dots&  \\
|{\rm{Cil}}_0\partial{\Gamma_+}[3]|&1& 8 & 18  & 16 &  0& 0 &  0& 0& \dots& \\
|{\rm{Cil}}_0\partial{\Gamma_+}[4]|&1& 10 & 30 &40 & 25 & 0 & 0 &  0& \dots& \\
|{\rm{Cil}}_0\partial{\Gamma_+}[5]|&1& 12 & 45 & 80 &  75& 36& 0& 0 &\dots&  \\
|{\rm{Cil}}_0\partial{\Gamma_+}[6]|&1& 14 & 63 &140  & 175 & 126 &  49& 0 & \dots& \\
~~~~~~~~~~\vdots &\vdots & \vdots  &\vdots  & \vdots & \vdots  &\vdots  & \vdots  & \vdots  & \ddots  &  \\

\end{array}
$$

$$
\begin{array}{l |l ccccccccc}
 &-1& 0 & 1 & 2 & 3 & 4 & 5 & 6 & \dots \\\hline
|{\rm{Cil}}_0{\Gamma_+}[-1]|&1& 0 & 0 & 0 & 0 & 0 & 0 & 0 & \dots&  \\
|{\rm{Cil}}_0{\Gamma_+}[0]|&1& 2 & 0 & 0 & 0 & 0 & 0 & 0 &\dots& \\
|{\rm{Cil}}_0{\Gamma_+}[1]|&1& 4 & 3 & 0 & 0 & 0 & 0 & 0 &\dots&  \\
|{\rm{Cil}}_0{\Gamma_+}[2]|&1& 6 & 9 &4  & 0 & 0 &  0& 0 &\dots&  \\
|{\rm{Cil}}_0{\Gamma_+}[3]|&1& 8 & 18  & 16 &  5& 0 &  0& 0& \dots& \\
|{\rm{Cil}}_0{\Gamma_+}[4]|&1& 10 & 30 &40 & 25 & 6 & 0 &  0& \dots& \\
|{\rm{Cil}}_0{\Gamma_+}[5]|&1& 12 & 45 & 80 &  75& 36& 7& 0 &\dots&  \\
|{\rm{Cil}}_0{\Gamma_+}[6]|&1& 14 & 63 &140  & 175 & 126 &  49& 8 & \dots& \\
~~~~~~~~~~\vdots &\vdots & \vdots  &\vdots  & \vdots & \vdots  &\vdots  & \vdots  & \vdots  & \ddots  &  \\

\end{array}
$$

\begin{obs}

Note that $a(n)=|\text{Cil}_0{\Gamma_+}[n]|_1$ is the sequence  A045943	of	triangular matchstick numbers:
$a(n)=3 n \frac{(n+1)}{2}$	(\url{https://oeis.org/A045943}).

Taking $a(n)=|\text{Cil}_0{\Gamma_+}[n+1]|_2$, we have the sequence  A210440, where $a(n)=2 n (n+1) \frac{(n+2)}{3}$ (\url{https://oeis.org/A210440}).
\end{obs}

\subsubsection{ The 2-cylinder}

Now, we give the third kind of cylinder we are considering in this work.

\begin{defi}
For each $[n] \in {\mathbf{\Gamma}_+}$, we introduce $\text{Cil}_2{\Gamma_+}[n]$, the augmented semi-simplicial set defined as the join:
$$\text{Cil}_2{\Gamma_+}[n]:={\Gamma_+}[n] \boxplus {\Gamma_+}[n].$$
Recall that, for each $m \in {\mathbf{\Gamma}_+}$, $\text{Cil}_2{\Gamma_+}[n]_m$ is constituted by pairs of the form $(\sigma ,\tau )\in {\Gamma_+}[n]_p\times {\Gamma_+}[n]_q$ with $p+q=m-1$. We can follow the notation established in previous subsections. Moreover, if $ \sigma \in {\Gamma_+}[n]_p$, $\tau \in {\Gamma_+}[n]_q$, $p+q=m-1,$ and $0\leq i\leq m$, the  face operator is given by
$$
d_i((\sigma,\tau)) =
\left\{
\begin{array}{lll}
 (d_i(\sigma), \tau) & \mbox{if $0\leq i \leq p $},\\ [1pc]
(\sigma, d_{i-p-1}(\tau))  & \mbox{ if $p+1 \leq i \leq m$}.
\end{array} \right.
$$
\end{defi}

\begin{ej}
We describe ${\rm{Cil}}_2{\Gamma_+}[1]$:

${\rm{Cil}}_2{\Gamma_+}[1]_{-1}=\{*\}$,

${\rm{Cil}}_2{\Gamma_+}[1]_{0}=\{(\emptyset, (0)),(\emptyset, (1)), ((0), \emptyset) , ((1), \emptyset)  \}=\{ 0\sp{\prime},1\sp{\prime}, 0,1\},$

${\rm{Cil}}_2{\Gamma_+}[1]_{1}=\{(\emptyset, (0,1)),((0), (0)), ((0), (1)), ((1), (0)), ((1), (1)),( (0,1), \emptyset ) \}\\=\{(0\sp{\prime},1\sp{\prime}), (0,0\sp{\prime}), (0,1\sp{\prime}),(1,0\sp{\prime}), (1,1\sp{\prime}), (0,1)\},$

${\rm{Cil}}_2{\Gamma_+}[1]_{2}=\{((0), (0,1)), ((1), (0,1)),((0,1), (0)), ((0,1), (1)))\}=\\
 \{(0, 0\sp{\prime},1\sp{\prime}), (1 ,0\sp{\prime},1\sp{\prime}),(0,1 ,0\sp{\prime}), (0,1,1\sp{\prime})\},$

 ${\rm{Cil}}_2{\Gamma_+}[1]_{3}=\{((0,1), (0,1))\}=\{(0, 1, 0\sp{\prime},1\sp{\prime}))\}.$

 And  ${\rm{Cil}}_2{\Gamma_+}[1]_{m}=\emptyset$ for $m\geq 4$.

\end{ej}

We have an obvious induced functor $\text{Cil}_2 \colon {\mathbf{\Gamma_+}} \to \Set, \, \text{Cil}_2([n])= \text{Cil}_2{\Gamma_+}[n]$ and, by Theorem \ref{ext-pr}, the colimit-preserving functor $(-)\widetilde{\vartriangleright} \text{Cil}_2 :\mathbf{Sets}^{\mathbf{\Gamma}_+^{op}}\to \mathbf{Sets}^{\mathbf{\Gamma}_+^{op}}$, making commutative the diagram
$$
\xymatrix{ \ar@{^(->}[d]_{\Yoneda} {\mathbf{\Gamma_+}} \ar[r]^{\text{Cil}_2} & \mathbf{Sets}^{\mathbf{\Gamma}_+^{op}}\\
\mathbf{Sets}^{\mathbf{\Gamma}_+^{op}}\ar[ru]_{(-)\widetilde{\vartriangleright} \text{Cil}_2} & }, $$

\begin{defi}
For any augmented semi-simplicial set $X \in \mathbf{Sets}^{\mathbf{\Gamma}_+^{op}}$, we define its \emph{2-cylinder} as
$$\text{Cil}_2(X):= X \widetilde{\vartriangleright} \text{Cil}_2$$ Similarly, for augmented semi-simplicial maps $\text{Cil}_2(f):= f \widetilde{\vartriangleright} \text{Cil}_2$.
\end{defi}
This functor $\text{Cil}_2$  has a right adjoint $\text{r}_{\text{Cil}_2} \colon \Set^{{\mathbf{\Gamma}_+^{op}}} \to \Set^{{\mathbf{\Gamma}_+^{op}}}.$

If we have an augmented semi-simplicial subset $X \subset \Gamma_+[n]$, we can take:

$\text{Cil}_2X=\{(\sigma, \tau) \in \text{Cil}_2{\Gamma_+}[n]\hspace{2pt}|\hspace{2pt} \text{there is }\hspace{2pt}\gamma \in X_m \hspace{2pt} \text{with}  \hspace{2pt}(\text{Im}(\sigma) \cup \text{Im}(\tau)) \subset \text{Im}(\gamma) \}.$

If we denote $\breve{\text{Cil}}_2{\Gamma_+}[n]={\text{Cil}_2}{\Gamma_+}[n] \setminus \text{Cil}_2\partial \Gamma_+[n]$, then it is not hard to prove the following result. Recall our convention that the combinatorial number $\binom{p}{q}$ is zero whenever $p<q$ or $q<0.$

\begin{prop} For any $n, m \geq -1$ we have
\begin{itemize}
\item[(i)] The number total of simplices in $\breve{\text{Cil}}_2{\Gamma_+}[n]$ is $$\sum_{i=-1}^{n} \binom{n+1}{i+1} \sum_{j=-1}^i \binom{i+1}{j+1}.$$
\item[(ii)] $|\breve{\text{Cil}}_2{\Gamma_+}[n]|_m=  \sum_{i=-1}^{n} \binom{n+1}{i+1} \binom{i+1}{m-n}.$
\end{itemize}
\end{prop}

\begin{proof} It suffices to count the pair of subsets $\{v_0, \cdots, v_p\}$, $\{w_0, \cdots, w_q\}$ of $\{0,1, \cdots, n\}=[n]$ such that their union is equal to $[n]$ and $(1+p)+(1+q)=(1+m)-(1+n)$.
\end{proof}

Using the proposition above, we can describe the cardinal sequences $|\breve{\text{Cil}}_2{\Gamma_+}[n]|$, for all $n$. Here we provide a list of the six first examples:
$$
\begin{array}{l |l ccccccccc}
 &-1& 0 & 1 & 2 & 3 & 4 & 5 & 6 & \dots \\\hline
|\breve{\text{Cil}}_2{\Gamma_+}[-1]|&1& 0 & 0 & 0 & 0 & 0 & 0 & 0 & \dots&  \\
|\breve{\text{Cil}}_2{\Gamma_+}[0]|&0& 2 & 1 & 0 & 0 & 0 & 0 & 0 &\dots& \\
|\breve{\text{Cil}}_2{\Gamma_+}[1]|&0& 0 &   4 &  4&  1 & 0 & 0 & 0 &\dots&  \\
|\breve{\text{Cil}}_2{\Gamma_+}[2]|&0& 0 & 0 & 8 &   12&   6&    1& 0 &\dots&  \\
|\breve{\text{Cil}}_2{\Gamma_+}[3]|&0& 0 &  0& 0 & 16  & 32 &  24 & 8  & \dots& \\
|\breve{\text{Cil}}_2{\Gamma_+}[4]|&0& 0 & 0 &0 &0 &  32 &80 &80 & \dots& \\
|\breve{\text{Cil}}_2{\Gamma_+}[5]|&0&  0& 0 &  0& 0 &0 & 64 &192&\dots&  \\
|\breve{\text{Cil}}_2{\Gamma_+}[6]|&0& 0 & 0 & 0 &  0& 0 & 0 &  128& \dots& \\
~~~~~~~~~~\vdots &\vdots & \vdots  &\vdots  & \vdots & \vdots  &\vdots  & \vdots  & \vdots  & \ddots  & \\

\end{array}
$$

This will enable the calculation of the sequences $|\text{Cil}_2\partial{\Gamma_+}[n]|$ and $|\text{Cil}_2\partial{\Gamma_+}[n]|$. We need the following result, which provides an easy algorithm.

\begin{prop}\label{2cilindroprop}  For any $n, m \geq -1$ we have:

\begin{itemize}

\item[(i)] $|\text{Cil}_2\partial{\Gamma_+}[n]|=\sum _{i=-1}^{n-1}\binom{n+1}{i+1}  |\breve{\text{Cil}}_2{\Gamma_+}[i]|$.

\item[(ii)] $|\text{Cil}_2{\Gamma_+}[n]|=|\text{Cil}_2\partial{\Gamma_+}[n]|+|\breve{\text{Cil}}_2{\Gamma_+}[n]|$.

\end{itemize}
\end{prop}

\bigskip
\begin{ej}

For instance, for $n=2$, we have

$|{\rm{Cil}}_2\partial{\Gamma_+}[2]|= |\breve{{\rm{Cil}}}_2{\Gamma_+}[-1]|+ \binom{3}{1} |\breve{{\rm{Cil}}}_2{\Gamma_+}[0]|+ \binom{3}{2} |\breve{{\rm{Cil}}}_2{\Gamma_+}[1]|$.

As $|\breve{{\rm{Cil}}}_2{\Gamma_+}[-1]|=(1, 0 , 0 , 0 , \cdots)$, $|\breve{{\rm{Cil}}}_2{\Gamma_+}[0]|=(0, 2 , 1 , 0 , 0 , 0 ,\cdots)$ and  $|\breve{{\rm{Cil}}}_2{\Gamma_+}[1]|=(0, 0 , 4 , 4 , 1 , 0 , 0 , 0 ,\cdots)$ we conclude that
$$|{\rm{Cil}}_2\partial{\Gamma_+}[2]|=(1, 6 , 15 , 12 , 3 , 0 , 0 , 0 , \cdots).$$
Finally, as $|\breve{{\rm{Cil}}}_2{\Gamma_+}[2]|=(0,0,0,8,12,6,1,0,0,\cdots)$ we have that $$|{\rm{Cil}}_2{\Gamma_+}[2]|=|{\rm{Cil}}_2\partial{\Gamma_+}[2]|+|\breve{{\rm{Cil}}}_2{\Gamma_+}[2]|=(1, 6 , 15 , 20 , 15 , 6 , 1 , 0 , \cdots).$$
\end{ej}

Using this simple algorithm, we can compute the sequence cardinals $|\text{Cil}_2\partial{\Gamma_+}[n]|$ and
$|\text{Cil}_2{\Gamma_+}[n]|$, for all $n$. Here, we provide a table for the first examples:

$$
\begin{array}{l |l ccccccccc}
 &-1& 0 & 1 & 2 & 3 & 4 & 5 & 6 & \dots \\\hline
|\text{Cil}_2\partial{\Gamma_+}[-1]|&0& 0 & 0 & 0 & 0 & 0 & 0 & 0 & \dots&  \\
|\text{Cil}_2\partial{\Gamma_+}[0]|&1& 0 & 0 & 0 & 0 & 0 & 0 & 0 &\dots& \\
|\text{Cil}_2\partial{\Gamma_+}[1]|&1& 4 & 2 & 0 & 0 & 0 & 0 & 0 &\dots&  \\
|\text{Cil}_2\partial{\Gamma_+}[2]|&1& 6 & 15 & 12  & 3 & 0&  0& 0 &\dots&  \\
|\text{Cil}_2\partial{\Gamma_+}[3]|&1& 8 & 28 & 56  & 54 & 24 &  4& 0 &\dots&  \\
|\text{Cil}_2\partial{\Gamma_+}[4]|&1& 10 & 45 &  120& 210 & 220 & 130  &40  & \dots& \\
|\text{Cil}_2\partial{\Gamma_+}[5]|&1& 12 & 66  & 220 & 495 & 792 & 860 & 600  & \dots& \\
|\text{Cil}_2\partial{\Gamma_+}[6]|&1& 14 & 91  &364  & 1001 & 2002& 3003& 3304 &\dots&  \\
|\text{Cil}_2\partial{\Gamma_+}[7]|&1& 16  & 120 &  560&  1820 & 4368 &8008  & 11440 & \dots& \\
~~~~~~~~~~\vdots &\vdots & \vdots  &\vdots  & \vdots & \vdots  &\vdots  & \vdots  & \vdots  & \ddots  & \\

\end{array}
$$

$$
\begin{array}{l |l ccccccccc}
 &-1& 0 & 1 & 2 & 3 & 4 & 5 & 6 & \dots \\\hline
|\text{Cil}_2{\Gamma_+}[-1]|&1& 0 & 0 & 0 & 0 & 0 & 0 & 0 & \dots&  \\
|\text{Cil}_2{\Gamma_+}[0]|&1& 2 & 1 & 0 & 0 & 0 & 0 & 0 &\dots& \\
|\text{Cil}_2{\Gamma_+}[1]|&1& 4 & 6 & 4 & 1 & 0 & 0 & 0 &\dots&  \\
|\text{Cil}_2{\Gamma_+}[2]|&1& 6 & 15 & 20 & 15 & 6 &  1& 0 &\dots& \\
|\text{Cil}_2{\Gamma_+}[3]|&1& 8 &28  &56  &70  &56  & 28 & 8 & \dots& \\
|\text{Cil}_2{\Gamma_+}[4]|&1& 10  & 45 & 120  & 210 & 252 & 120 &45  &  \dots& \\
|\text{Cil}_2{\Gamma_+}[5]|&1&  12 & 66  & 220  & 495 & 792 &495 & 220 &\dots&  \\
|\text{Cil}_2{\Gamma_+}[6]|&1&  14& 91 & 364 &  1001&  2002& 3003  & 3432 & \dots& \\
~~~~~~~~~~\vdots &\vdots & \vdots  &\vdots  & \vdots & \vdots  &\vdots  & \vdots  & \vdots  & \ddots  &  \\

\end{array}
$$

\begin{obs} Observe that the functor $\text{Cil}_2  \colon \Set^{{\mathbf{\Gamma}_ +^{op}}} \to \Set^{{\mathbf{\Gamma}_+^{op}}}$ and the functor $ \text{Dup}\colon \Set^{{\mathbf{\Gamma}_+^{op}}} \to \Set^{{\mathbf{\Gamma}_+^{op}}}$ given as $$\text{Dup}(X):=X\boxplus X$$ \noindent verify that
$\text{Dup} \Yoneda= \text{Cil}_2  \Yoneda $ ($\Yoneda$ is the Yoneda functor). However, these functors are different. For instance, if $X=\partial{\Gamma_+}[2]$, then one has that $|\text{Cil}_2(\partial{\Gamma_+}[2])|=(1,6,15,12,3,0, \cdots)$ and $|\text{Dup}(\partial{\Gamma_+}[2])|=(1,6,15,18,9,0, \cdots)$. This implies that $\text{Dup}$ does not preserve colimits.
\end{obs}

\begin{obs} For the functors $\text{Cil}_0, \text{Cil}, \text{Cil}_2, \text{Dup} \colon \Set^{{\mathbf{\Gamma}_ +^{op}}} \to \Set^{{\mathbf{\Gamma}_+^{op}}}$, there are natural transformacions
$$\text{Cil}_0 \subset \text{Cil} \subset \text{Cil}_2 \subset  \text{Dup}. $$
\end{obs}

\begin{obs}
Note that, for $n\geq 0$, $a(n)=|\text{Cil}{\Gamma_+}[n-1]|_1$ is the sequence  AA000384 of	hexagonal numbers:
$a(n) = n (2n-1)$	(\url{https://oeis.org/AA000384}).

Taking $a(n)=|\text{Cil}{\Gamma_+}[n]|_2$, we have the sequence A002492 which  is the sum of the first even squares,   where 		$a(n) = 2 n (n+1) \frac{(2 n+1)}{3}$ (\url{https://oeis.org/A002492}).

\end{obs}

\subsection{Barycentric subdivision of an augmented semi-simplicial set}\label{geometricsubdivisions}

In this subsection we describe the barycentric subdivision of an augmented semi-simplicial set. For any $[n]\in \mathbf{\Gamma_+}$, we first introduce the barycentric subdivision of ${\Gamma_+}[n]$, denoted by $\text{Sd}{\Gamma_+}[n]$, as the following augmented semi-simplicial set: If $-1\leq m\leq n$, an element $a=(\varphi _{-1},\varphi _0,\cdots, \varphi _m)\in \text{Sd}{\Gamma_+}[n]_m$ is just a chain of composable morphisms in $\mathbf{\Gamma_+}$, of the form
$$\emptyset =[-1]\stackrel{\varphi _{-1}}{\longrightarrow }[k_0]\stackrel{\varphi _0}{\longrightarrow }[k_1]\stackrel{\varphi _1}{\longrightarrow }\cdots
\stackrel{\varphi _{m-1}}{\longrightarrow }[k_m]\stackrel{\varphi _m}{\longrightarrow }[k_{m+1}]=[n]$$
\noindent where $0\leq k_0<k_1<\cdots<k_m\leq n$. We set $\text{Sd}{\Gamma_+}[n]_m=\emptyset $, if $m>n.$

Now, if $m\geq 0$ and $0\leq i\leq m,$ then the face operator $d_i:\text{Sd}{\Gamma_+}[n]_m\rightarrow \text{Sd}{\Gamma_+}[n]_{m-1}$ is given as
$$d_i(a)=d_i((\varphi _{-1},\varphi _0,\cdots, \varphi _m)):=(\varphi _{-1},...,\varphi _{i-2},\varphi _i\circ \varphi _{i-1},\varphi _{i+1},\cdots ,\varphi _m)$$
Observe that $\text{Sd}{\Gamma_+}[n]_{-1}=\{\emptyset= (\emptyset  \to [n])\}.$ In particular, $\text{Sd}{\Gamma_+}[-1]_{-1}=\{\emptyset= (\emptyset  \to [-1])\};$ moreover $\text{Sd}{\Gamma_+}[-1]_{m}=\emptyset $, for all $m$. We also have
$\text{Sd}{\Gamma_+}[0]_{-1}=\{\emptyset= (\emptyset  \to [-1])\},$
$\text{Sd}{\Gamma_+}[0]_{0}=\{ \emptyset  \to [0]\}$ and $\text{Sd}{\Gamma_+}[0]_{m}= \emptyset $, for all $m\geq 1.$
This implies that
\begin{equation}
\label{initialsubdivision}
\text{Sd}{\Gamma_+}[-1]= {\Gamma_+}[-1]\hspace{8pt}\mbox{and}\hspace{8pt} \text{Sd}{\Gamma_+}[0]= {\Gamma_+}[0].
\end{equation}

Then, we can consider the canonical functor
$\text{Sd} \colon {\mathbf{\Gamma}_+}\to \Set^{{\mathbf{\Gamma}_+^{op}}}, \quad \text{Sd} ([n])=\text{Sd}{\Gamma_+}[n]$ and, by Theorem \ref{ext-pr}, the colimit-preserving functor $(-)\widetilde{\vartriangleright}\text{Sd} :\mathbf{Sets}^{\mathbf{\Gamma}_+^{op}}\to \mathbf{Sets}^{\mathbf{\Gamma}_+^{op}}$, making commutative the diagram
$$
\xymatrix{\ar[d]_{\text{Y}} {\mathbf{\Gamma}_+} \ar[r]^{\text{Sd} } & \Set^{{\mathbf{\Gamma}_+^{op}}}\\
\Set^{{\mathbf{\Gamma}_+^{op}}}  \ar[ru]_{(-)\widetilde{\vartriangleright}\text{Sd}}&}
$$ and it has a right adjoint
$$\text{r}_{\text{Sd}} \colon \Set^{{\mathbf{\Gamma}_+^{op}}} \to \Set^{{\mathbf{\Gamma}_+^{op}}}$$

\begin{defi}\label{subd-semi}
For any augmented semi-simplicial set $X\in \Set^{{\mathbf{\Gamma}_+^{op}}}$ we define its \emph{barycentric subdivision} as
$$\text{Sd} (X) =X \tilde{ \vartriangleright} \text{Sd} $$
Similarly, for augmented semi-simplicial maps, $\text{Sd} (f) =f \tilde{ \vartriangleright} \text{Sd}.$
\end{defi}

\medskip
For the next result, we recall that $S_+[n-1]=\Gamma_+[n]\setminus \{\iota _n\},$ where $\iota_n:[n]\rightarrow [n]$ is the identity map.

\begin{prop}\label{ConoSubdivision} Consider the cone funtors $\text{Con}_l, \text{Con}_r  \colon \Set^{{\Gamma_+^{op}}}  \to \Set^{{\Gamma_+^{op}}}$.
Then there are canonical isomorphisms
$$\text{Con}_l (\text{Sd} (S_+[n-1]))\cong \text{Sd}{\Gamma_+}[n]\cong \text{Con}_r (\text{Sd} (S_+[n-1])$$
for $k\geq 0.$
\end{prop}

\begin{obs}The isomorphisms above can also be written as:
$$\begin{array}{lcl}
{\Gamma_+}[0] \boxplus (\text{Sd} (S_+[n-1])) & = & {\Gamma_+}[0] \boxplus( S_+[n-1]\tilde{\vartriangleright} \text{Sd}) \\
 & \cong & ({\Gamma_+}[0] \boxplus S_+[n-1]) \tilde{\vartriangleright} \text{Sd} \\
 & \cong & {\Gamma_+}[n] \tilde{\vartriangleright} \text{Sd} \\
 & \cong & ( S_+[n-1] \boxplus {\Gamma_+}[0]) \tilde{\vartriangleright} \text{Sd}\Gamma_+ \\
 & \cong  & (\text{Sd} (S_+[n-1]))\boxplus  {\Gamma_+}[0].
\end{array}$$
\end{obs}

\subsection{Cylinders for integer sequences}

\subsubsection{The standard cylinder for integer sequences}

Associated with the augmented semi-simplicial sets $\text{Cil} \Gamma_+[n]$, $\text{Cil} \partial \Gamma_+[n]$, we consider the following augmented  integer sequences:

\begin{defi}\label{cilinders} The augmented sequence $ \breve {\text{cil}}{\gamma_+[n]}\in {\mathbf{Z}}^{{\mathbb{N}_+^{op}}}$ is given by
\begin{equation}
\label{brevecil}
\breve{\text{cil}}{\gamma_+}[n]_m:=
\left\{
\begin{array}{lll}
n+2 & \mbox{if $m=n $},\\ [1pc]
n+1 & \mbox{ if $m=n+1$}\\ [1pc]
0 & \mbox{ if $m\not \in \{n, n+1\}$}
\end{array} \right.
\end{equation}
and the augmented matrix $\breve{\text{cil}}=(\breve{\text{cil}}_{nm})$ by $\breve{\text{cil}}_{nm}:=\breve{\text{cil}}{\gamma_+}[n]_m.$

The augmented  sequences  $\text{cil}\partial{\gamma_+}[n],  \text{cil}{\gamma_+}[n] \in {\mathbf{Z}}^{{\mathbb{N}_+^{op}}}$ are defined as
\begin{equation}
\label{partialcil}
 \text{cil}\partial{\gamma_+}[n]:= \sum_{j=-1}^{n-1} \binom{1+n}{1+j}  \breve{\text{cil}}{\gamma_+}[j]
\end{equation}
\begin{equation}
\label{cil}
 \text{cil}{\gamma_+}[n]:=\sum_{j=-1}^{n} \binom{1+n}{1+j}  \breve{\text{cil}}{\gamma_+}[j]
\end{equation}
and the augmented matrices $ \text{cil}\partial= (\text{cil}\partial_{nm})$ , $ \text{cil}=( \text{cil}_{nm})$ by
$\text{cil}\partial_{nm}:= \text{cil}\partial{\gamma_+}[n]_m,$  $ \text{cil}_{nm}:= \text{cil}{\gamma_+}[n]_m$.
\end{defi}

Recall that, for $k \in \mathbb{N}_+$,
${\mathbf{1}}_{k}$ denotes the augmented sequence given by $({\mathbf{1}}_{k})_i=\delta_{k, i}$.
For $k,l \in  \mathbb{N}_+$, we also consider the augmented matrices ${\mathbf{1}}_{k,l}$ where $({\mathbf{1}}_{k,l})_{i,j}:= \delta_{k, i} \delta_{l, j}$

\begin{prop} The augmented matrices $\breve{\text{cil}}$,   $\text{cil}\partial$, $\text{cil}$ can be described as follows:
\begin{itemize}
  \item[(i)] $\breve{\text{cil}}_{nm}= (n+2)   \delta_{n, m}+(n+1) \delta_{n+1, m}$
  \item[(ii)]  $\text{cil}\partial_{nm}= \sum_{j=-1}^{n-1} \binom{1+n}{1+j}  ( (j+2)   \delta_{j, m}+(j+1) \delta_{j+1, m})$
  \item[(iii)]  $\text{cil}_{nm}= \sum_{j=-1}^{n} \binom{1+n}{1+j}  ( (j+2)   \delta_{j, m}+(j+1) \delta_{j+1, m})$
\end{itemize}
\end{prop}

\begin{proof}  (iii)  Since
 $\breve{\text{cil}}{\gamma_+}[j]_m=( (j+2)   \delta_{j, m}+(j+1) \delta_{j+1, m})$
 we have
 $\text{cil}{\gamma_+}[n]_m= \sum_{j=-1}^{n} \binom{1+n}{1+j}  \breve{\text{cil}}{\gamma_+}[j]_m= \sum_{j=-1}^{n} \binom{1+n}{1+j}  ( (j+2)   \delta_{j, m}+(j+1) \delta_{j+1, m})$.

\end{proof}

\begin{prop}\label{cilCil}
The following equalities hold true:
\begin{equation}
\label{Cilcil+}
 \text{cil}\gamma_+[n]=|\text{Cil} ({\Gamma_+}[n])|, \quad   \text{cil}=|\text{Cil}|
\end{equation}
 where $\text{Cil}$ is the co-semi-simplicial object given in subsection \ref{geometriccylinders}
\end{prop}

\begin{proof} It is a direct consequence of Proposition \ref{sucesioncardinalCil} and Definition \ref{cilinders}.
\end{proof}

\begin{defi} Given an augmented integer sequence $ a \in {\bf Z}^{\mathbb{N}_+^{op}} $, the \emph{standard cylinder of
$a$} is the sequence obtained as $\text{cil} (a):=a \tilde{\triangleright} \text{cil}$. This construction gives rise to a functor
$$\text{cil} \colon {\bf Z}^{\mathbb{N}_+^{op}} \to {\bf Z}^{\mathbb{N}_+^{op}}$$
\end{defi}

\begin{obs}
Observe that $\text{bin}^{-1} \cdot  \text{cil}=\breve{\text{cil}}$ and therefore $a \tilde{\triangleright} \text{cil}$ has sense, for all $a \in {\bf Z}^{\mathbb{N}_+^{op}}.$
\end{obs}

\begin{cor}  Given an augmented integer sequence $a \in {\bf Z}^{\mathbb{N}_+^{op}}$, its standard cylinder can be described as
$ \text{cil} (a)= a \cdot \breve{\text{cil}}.$
\end{cor}

\begin{obs} Note that the functor $\text{cil} \colon {\bf Z}^{\mathbb{N}_+^{op}} \to {\bf Z}^{\mathbb{N}_+^{op}}$ has an inverse functor $\text{cil}^{-1} \colon {\bf Z}^{\mathbb{N}_+^{op}} \to {\bf Z}^{\mathbb{N}_+^{op}}$,
$ \text{cil}^{-1} (a)= a \cdot (\breve{\text{cil}})^{-1} $,  $ a \in {\bf Z}^{\mathbb{N}_+^{op}} $. For $k \in \mathbb{Z}$,  one can also consider the iteration functor  $\text{cil}^{k}$ given by $ \text{cil}^{k} (a)= a \cdot (\breve{\text{cil}})^{k} $,  $ a \in {\bf Z}^{\mathbb{N}_+^{op}} $.
\end{obs}

\subsubsection{The 0-cylinder for integer sequences}
This subsection will follow a similar structure to the previous one for standard cylinders of sequences. Indeed,
associated with $\text{Cil}_0  \Gamma_+[n]$, $\text{Cil}_0  \partial \Gamma_+[n]$, we consider the following augmented integer sequences;

\begin{defi}\label{cilinderscero} The sequence  $ \breve {\text{cil}}_0 \gamma_+[n]\in {\mathbf{Z}}^{{\mathbb{N}_+^{op}}}$ is given by
\begin{equation}
\label{brevecil}
\breve{\text{cil}}_0{\gamma_+}[n]_m= (n+2) \delta_{n,m}
\end{equation}
and the augmented matrix $\breve{\text{cil}}_0=((\breve{\text{cil}}_0)_{nm})$ is given by $(\breve{\text{cil}}_0)_{nm}=\breve{\text{cil}}_0{\gamma_+}[n]_m.$

The sequences  $  \text{cil}_0\partial{\gamma_+}[n],  \text{cil}_0{\gamma_+}[n] \in {\mathbf{Z}}^{{\mathbb{N}_+^{op}}}$ are defined by
\begin{equation}
\label{partialcil}
 \text{cil}_0\partial{\gamma_+}[n]= \sum_{j=-1}^{n-1} \binom{1+n}{1+j}  \breve{\text{cil}}_0{\gamma_+}[j]
\end{equation}
\begin{equation}
\label{cil}
 \text{cil}_0{\gamma_+}[n]= \sum_{j=-1}^{n} \binom{1+n}{1+j}  \breve{\text{cil}}_0{\gamma_+}[j]
\end{equation}
and the augmented matrices $ \text{cil}_0\partial= ((\text{cil}_0\partial)_{nm})$ , $ \text{cil}_0=( (\text{cil}_0)_{nm})$ are given by
 $ (\text{cil}_0\partial)_{nm}= \text{cil}_0\partial{\gamma_+}[n]_m,$  $ (\text{cil}_0)_{nm}= \text{cil}_0{\gamma_+}[n]_m.$
\end{defi}

\begin{prop} The augmented matrices $\breve{\text{cil}}_0$,   $\text{cil}_0\partial$, $\text{cil}_0$ can be described as follows:
\begin{itemize}
  \item[(i)] $(\breve{\text{cil}}_0)_{nm}= (n+2)   \delta_{n, m},$
  \item[(ii)]  $(\text{cil}_0\partial)_{nm}= \sum_{j=-1}^{n-1} \binom{1+n}{1+j}   (j+2)   \delta_{j, m},$
  \item[(iii)]  $(\text{cil}_0)_{nm}= \sum_{j=-1}^{n} \binom{1+n}{1+j}  (j+2)   \delta_{j, m}.$
\end{itemize}
\end{prop}

\begin{proof} It is a routine check.

\end{proof}

\begin{prop}\label{cilCilcero}
For the augmented simplicial set  $\text{Cil}_0 {\Gamma_+}[n]$  we have
\begin{equation}
\label{Cilcilcero}
 \text{cil}_0\gamma_+[n]=|\text{Cil}_0 {\Gamma_+}[n]|, \quad   \text{cil}_0=|\text{Cil}_0|
\end{equation}
 where $\text{Cil}_0$ is the co-semi-simplicial object given in subsection \ref{geometriccylinders}
\end{prop}

\begin{proof} It is a consequence of Proposition \ref{Cilcerocal} and Definition \ref{cilinderscero}.

\end{proof}

We define the notion of a 0-cylinder of an augmented sequence analogously to that of its standard cylinder.

\begin{defi}  The $0$-cylinder of an augmented integer sequence $a \in {\bf Z}^{\mathbb{N}_+^{op}} $ is defined as the tilde-triangle product
$\text{cil}_0 (a):=a \tilde{\triangleright} \text{cil}_0.$ This construction gives rise to a functor
$$\text{cil}_0 \colon {\bf Z}^{\mathbb{N}_+^{op}} \to {\bf Z}^{\mathbb{N}_+^{op}}$$
\end{defi}

In this case, as $\text{bin}^{-1} \cdot  \text{cil}_0=\breve{\text{cil}}_0$, this definition is well-given.

\begin{cor}  Given an augmented integer sequence $a \in {\bf Z}^{\mathbb{N}_+^{op}} $ its $0$-cylinder can be described as
$$\text{cil}_0 (a)=a \cdot \breve{\text{cil}}_0 .$$
\end{cor}

\begin{obs} Note that the functor $\text{cil}_0 \colon {\bf Z}^{\mathbb{N}_+^{op}} \to {\bf Z}^{\mathbb{N}_+^{op}}$ also has an inverse functor $\text{cil}_0^{-1} \colon {\bf Z}^{\mathbb{N}_+^{op}} \to {\bf Z}^{\mathbb{N}_+^{op}}$,
$ \text{cil}_0^{-1} (a)= a \cdot (\breve{\text{cil}}_0)^{-1} $,  $ a \in {\bf Z}^{\mathbb{N}_+^{op}} $. For $k \in \mathbb{Z}$,  one can also consider the iteration functor $\text{cil}_0^{k}$ given by $ \text{cil}_0^{k} (a)= a \cdot (\breve{\text{cil}}_0)^{k} $,  $ a \in {\bf Z}^{\mathbb{N}_+^{op}} $. In this case,
$$ (\text{cil}_0^{k} (a))_m= a_m (2+m)^k. $$
\end{obs}

\subsubsection{The 2-cylinder for integer sequences}

Again, following the same structure of previous subsection, and associated with the augmented semi-simplicial sets $\text{Cil} \Gamma_+[n]$, $\text{Cil} \partial \Gamma_+[n]$, we can consider the following integer sequences:

\begin{defi}\label{cilindersdos} The augmented sequence $ \breve {\text{cil}}_2 \gamma_+[n]\in {\mathbf{Z}}^{{\mathbb{N}_+^{op}}}$ is given by
\begin{equation}
\label{brevecil}
\breve {\text{cil}}_2 \gamma_+[n]_m:=  \sum_{i=-1}^{n} \binom{n+1}{i+1} \binom{i+1}{m-n}
\end{equation}
and the augmented sequences  $\text{cil}_2\partial{\gamma_+}[n],  \text{cil}_2{\gamma_+}[n] \in {\mathbf{Z}}^{{\mathbb{N}_+^{op}}}$ by
\begin{equation}
\label{partialcil}
 \text{cil}_2\partial{\gamma_+}[n]:= \sum_{j=-1}^{n-1} \binom{1+n}{1+j}  \breve{\text{cil}}{\gamma_+}[j]
\end{equation}
\begin{equation}
\label{cil}
 \text{cil}_2{\gamma_+}[n]:= \sum_{j=-1}^{n} \binom{1+n}{1+j}  \breve{\text{cil}}{\gamma_+}[j]
\end{equation}
\end{defi}

\begin{prop} The augmented matrices $\breve{\text{cil}}_2$,   $\text{cil}_2\partial$, $\text{cil}_2$ can be described as follows:
\begin{itemize}
  \item[(i)] $(\breve{\text{cil}}_2)_{nm}= \sum_{i=-1}^{n} \binom{n+1}{i+1} \binom{i+1}{m-n},$
  \item[(ii)]  $(\text{cil}_2\partial)_{nm}= \sum_{j=-1}^{n-1}  \sum_{i=-1}^{j}\binom{1+n}{1+j}  \binom{j+1}{i+1} \binom{i+1}{m-j},$
  \item[(iii)]  $(\text{cil}_2)_{nm}= \sum_{j=-1}^{n}  \sum_{i=-1}^{j}\binom{1+n}{1+j}  \binom{j+1}{i+1} \binom{i+1}{m-j}.$
\end{itemize}
\end{prop}

\begin{proof}  (iii)  Since
 $\breve{\text{cil}}_2{\gamma_+}[j]_m=\sum_{i=-1}^{j} \binom{j+1}{i+1} \binom{i+1}{m-j}$,
 we have

 $\text{cil}_2{\gamma_+}[n]_m= \sum_{j=-1}^{n} \binom{1+n}{1+j}  \breve{\text{cil}}_2{\gamma_+}[j]_m= \sum_{j=-1}^{n} \binom{1+n}{1+j}  \sum_{i=-1}^{j} \binom{j+1}{i+1} \binom{i+1}{m-j}$

\end{proof}

\begin{prop}\label{cilCildos}
For the augmented simplicial set  $\text{Cil}_2 {\Gamma_+}[n]$  we have
\begin{equation}
\label{Cilcildos}
 \text{cil}_2\gamma_+[n]=|\text{Cil}_2 {\Gamma_+}[n]|, \quad   \text{cil}_2=|\text{Cil}_2|
\end{equation}
 where $\text{Cil}_2$ is the co-semi-simplicial object given in subsection \ref{geometriccylinders}
\end{prop}

\begin{proof} It is a consequence of Proposition \ref{2cilindroprop} and Definition \ref{cilindersdos}.

\end{proof}

\begin{defi} The \emph{2-cylinder} of an augmented integer sequence $ a \in {\bf Z}^{\mathbb{N}_+^{op}} $,
is defined as the sequence $\text{cil}_2 (a)=a \tilde{\triangleright} \text{cil}_2$. This construction gives us a functor
$$\text{cil}_2 \colon {\bf Z}^{\mathbb{N}_+^{op}} \to {\bf Z}^{\mathbb{N}_+^{op}}$$
\end{defi}

As $\text{bin}^{-1} \cdot  \text{cil}_2=\breve{\text{cil}}_2$ this is well defined. Moreover

\begin{cor}  The 2-cylinder of $a \in {\bf Z}^{\mathbb{N}_+^{op}}$ can be described as
$$\text{cil}_2 (a)= a \cdot \breve{\text{cil}}_2 .$$
\end{cor}

\begin{obs} The functor $\text{cil}_2 \colon {\bf Z}^{\mathbb{N}_+^{op}} \to {\bf Z}^{\mathbb{N}_+^{op}}$ has an inverse functor $\text{cil}_2^{-1} \colon {\bf Z}^{\mathbb{N}_+^{op}} \to {\bf Z}^{\mathbb{N}_+^{op}}$ given as
$\text{cil}_2^{-1} (a)=a \cdot (\breve{\text{cil}}_2)^{-1} $, for all $ a \in {\bf Z}^{\mathbb{N}_+^{op}} $. For $k \in \mathbb{Z}$,  one can also consider the iteration functor  $\text{cil}_2^{k}$ given by $ \text{cil}_2^{k} (a)= a \cdot (\breve{\text{cil}}_2)^{k} $, for all  $a\in {\bf Z}^{\mathbb{N}_+^{op}} $.
\end{obs}

\subsection{Subdivision for integer sequences}\label{subdivisionofintegers}\label{subdicsequences}

Now, we analyse (barycentric) subdivisions for sequences. Recall that, for $n\geq 0$ we have $[n]=\{0,\cdots, n\}$ whereas $[-1]=\emptyset$.
As the empty set is always included, the cardinal of the power set $\mathcal{P}([n])$ is $2^{n+1}$.


For all $n, p \in {\mathbb{N}_+}$, we consider:
$$\text{Cad}^+[n]_p:=\{\emptyset=N_{-1}\subset N_0\subset N_1 \subset \cdots \subset N_p \subseteq [n]\hspace{2pt}:\hspace{2pt}  N_{i}\not = N_{i+1}, -1\leq i\leq p-1\}  \}$$
$$\breve{\text{Cad}}^+[n]_p:=\{\emptyset=N_{-1}\subset N_0\subset N_1 \subset \cdots \subset N_p = [n]\hspace{2pt}:\hspace{2pt}  N_{i}\not = N_{i+1}, -1\leq i\leq p-1\}  \}$$
$$\text{Cad}[n]_p:=\{\emptyset=N_{-1}\subseteq N_0\subset N_1 \subset \cdots \subset N_p \subseteq [n]\hspace{2pt}:\hspace{2pt}  N_{i}\not = N_{i+1}, 0\leq i\leq  p-1\}  \}$$
where the notation Cad comes from the Spanish word ``Cadena'', which means chain.
Taking cardinals, we obtain the augmented sequences
$$\text{cad}^+[n]_p:=|\text{Cad}^+[n]_p|, \quad \breve{\text{cad}}^+[n]_p:=|\breve{\text{Cad}}^+[n]_p|,\quad
\text{cad}[n]_p:=|\text{Cad}[n]_p|$$
\noindent and the corresponding matrices in $({{\bf Z}}^{{\mathbb{N}}_+^{op}})^{\mathbb{N}_+}$:
$$\text{cad}^+:=\begin{pmatrix} \text{cad}^+[-1] \\ \vdots \\ \text{cad}^+[n] \\ \vdots\end{pmatrix},\quad \breve{\text{cad}}^+:=\begin{pmatrix}\breve{\text{cad}}^+[-1] \\ \vdots\\ \breve{\text{cad}}^+[n] \\ \vdots\end{pmatrix}, \quad \text{cad}:=\begin{pmatrix} \text{cad}[-1) \\ \vdots\\ \text{cad}[n] \\ \vdots\end{pmatrix}$$

\begin{lem} \label{relacad}  The cone functor $\mbox{con}:{\mathbf{Z}}^{{\mathbb{N}_+^{op}}}\to {\mathbf{Z}}^{{\mathbb{N}_+^{op}}}$ satisfies the following relation:
$$\text{con} (\text{cad}^+[n])=\text{cad}[n]=\text{cad}^+[n]\boxplus {\bf c} = \text{cad}^+[n] \cdot R({\bf c}) $$
\noindent and, therefore, $ \text{cad}= \text{cad}^+ \cdot R({\bf c}).$
\end{lem}

\begin{teo} \label{relacadStirling2} The following equality holds true:
$$\breve{\text{cad}}^+[n]_p=(p+1)! \Stirling2{n+1}{ p+1}.$$
\end{teo}

\begin{proof}
Observe that, by definition of $\breve{\text{Cad}}^+[n]_p$, the elements in $$\{N_0, N_1\setminus N_0, \cdots, N_p\setminus N_{p-1}\}$$ \noindent are non empty and disjoint subsets. Consider the set $\text{Partitions}[n]_p$ consisting of all sets of the form $\{S_0, \cdots, S_p\}$ where every $S_i$ is not empty, $S_i\cap S_j=\emptyset$ for all $i\not=j$, and $S_0\cup \cdots \cup S_p=[n]$. Then, from the definition of Stirling number of second kind (see the introduction section \ref{Stirling}), it follows that
$$|\text{Partitions}[n]_p|=\Stirling2{n+1}{ p+1}.$$
There is a canonical surjective map
$$\text{Dif} \colon  \breve{\text{Cad}}^+[n]_p\to \text{Partitions}[n]_p$$
\noindent sending each chain $\emptyset=N_{-1}\subset N_0\subset N_1 \subset \cdots \subset N_{p}$ to $\{N_0, N_1\setminus N_0, \cdots, N_p\setminus N_{p-1}\}$. Since, in order to construct all the strictly increasing chains associated to a partition, we have to take into account all possible permutations of the members of this partition, we obtain:
$$\breve{\text{cad}}^+[n]_p=(p+1)! \Stirling2{n+1}{ p+1}.$$
\end{proof}

\begin{prop} The following properties hold true:
\begin{itemize}
\item [(i)] If $k<p$, then $\breve{\text{cad}}^+[k]_p=0$.
\item [(ii)] $\text{cad}^+[n]_p=\sum_{k=-1}^n \binom{n+1}{k+1} \breve{\text{cad}}^+[k]_p=\sum_{k=p}^n  \binom{n+1}{k+1} \breve{\text{cad}}^+[k]_p$.
\item [(iii)] $\text{cad}^+=\text{bin} \cdot \breve{\text{cad}}^+$ and $\text{bin}^{-1}\cdot {\text{cad}}^+=\breve{\text{cad}}^+$.
\end{itemize}
  \end{prop}

\begin{proof} (i) Just observe that, if $k<p$, then $\breve{\text{Cad}}^+[k]_p=\emptyset$.

(ii) For every finite set $S=\{s_0, \cdots, s_k\}$, we can obviously define $\breve{\text{Cad}}^+[S]_p$ so that there is a bijective correspondence between $\breve{\text{Cad}}^+[S]_p$ and $\breve{ \text{Cad}}^+[k]_p$ ($|S|=k+1$).
We can consider the bijection
$$\text{distribute} \colon \text{Cad}^+[n]_p \longrightarrow \bigsqcup_{S \subset [n], |S| \geq (p+1)}  \breve{\text{Cad}} ^+[S]_p $$
\noindent where each chain $\emptyset=N_{-1}\subset N_0\subset N_1 \subset \cdots \subset N_p \subseteq [n])$ is carried to
$$ (\emptyset=N_{-1}\subset N_0\subset N_1 \subset \cdots \subset N_p=S) \in  \breve{\text{Cad}} ^+[S]_p. $$

Therefore, we obtain
 $$\text{cad}^+[n]_p =|\text{Cad}^+[n]_p|=  \sum_{k=p}^n \binom{n+1}{k+1} |\breve{\text{Cad}}^+[k]_p|=\sum_{k=p}^n \binom{n+1}{k+1} \breve{\text{cad}}^+[k]_p.$$

(iii) This item follows from (i), (ii) and the definition of matrix multiplication.
\end{proof}

Using the proposition above, one can easily compute the coefficients in the augmented  matrices $\text{cad}^+$, $\breve{\text{cad}}^+$ and $\text{cad}$. Here, we are displaying the first rows of such matrices:
$$
\begin{array}{l | lcccccccc}
\text{cad}^+ &-1& 0 & 1 & 2 & 3 & 4 & 5 & 6 & 7\\\hline
\text{cad}^+[-1]&1& 0 & 0 & 0 & 0 & 0 & 0 & 0 & \cdots \\
\text{cad}^+[0]&1& 1 & 0 & 0 & 0 & 0 & 0 & 0 & \cdots \\
\text{cad}^+[1]&1& 3 & 2 & 0 & 0 & 0 & 0 & 0 & \cdots \\
\text{cad}^+[2]&1& 7 & 12 & 6 & 0 & 0 & 0 & 0 & \cdots \\
\text{cad}^+[3]&1& 15 & 50 & 60 & 24 & 0 & 0 & 0 & \cdots \\
\text{cad}^+[4]&1& 31 & 180 & 390 & 360 & 120 & 0 & 0 & \cdots \\
\text{cad}^+[5]&1& 63 & 602 & 2100 & 3360 & 2520 & 720 & 0 & \cdots \\
\text{cad}^+[6]&1& 127 & 1932 & 10206 & 25200 & 31920 & 20160 & 5040 & \cdots \\
\vdots &\vdots & \vdots  & \vdots  & \vdots  & \vdots & \vdots  & \vdots  & \vdots  & \ddots\\
\end{array}
$$

$$
\begin{array}{l | ccccccccc}
 \breve{\text{cad}}^+ & -1 & 0 & 1 & 2 & 3 & 4 & 5 & 6 & 7 \\ \hline
 \breve{\text{cad}}^+[-1]&1 & 0 & 0 & 0 & 0 & 0 & 0 & 0 & \cdots \\
\breve{\text{cad}}^+[0]& 0 & 1 & 0 & 0 & 0 & 0 & 0 & 0 & \cdots \\
\breve{\text{cad}}^+[1]& 0 & 1 & 2 & 0 & 0 & 0 & 0 & 0 & \cdots \\
\breve{\text{cad}}^+[2]& 0 & 1 & 6 & 6 & 0 & 0 & 0 & 0 & \cdots \\
\breve{\text{cad}}^+[3]& 0 & 1 & 14 & 36 & 24 & 0 & 0 & 0 & \cdots \\
\breve{\text{cad}}^+[4]&0 & 1 & 30 & 150 & 240 & 120 & 0 & 0 & \cdots\\
\breve{\text{cad}}^+[5]& 0 & 1 & 62 & 540 & 1560 & 1800 & 720 & 0 & \cdots \\
\breve{\text{cad}}^+[6]& 0 & 1 & 126 & 1806 & 8400 & 16800 & 15120 & 5040 & \cdots \\
\vdots & \vdots & \vdots & \vdots & \vdots & \vdots & \vdots & \vdots & \vdots & \ddots\\
\end{array}
$$

$$
\begin{array}{l | ccccccccc}
\text{cad} &-1& 0 & 1 & 2 & 3 & 4 & 5 & 6 & 7 \\\hline
\text{cad}[-1]&1& 1 & 0 & 0 & 0 & 0 & 0 & 0 & \cdots \\
\text{cad}[0]&1& 2 & 1 & 0 & 0 & 0 & 0 & 0 & \cdots \\
\text{cad}[1]&1& 4 & 5 & 2 & 0 & 0 & 0 & 0 & \cdots \\
\text{cad}[2]&1& 8 & 19 & 18 & 6 & 0 & 0 & 0 & \cdots \\
\text{cad}[3]&1& 16 & 65& 110 & 84 & 24 & 0 & 0 & \cdots \\
\text{cad}[4]&1& 32 & 211 & 570 & 750 & 480 & 120 & 0 & \cdots \\

\vdots &\vdots & \vdots  & \vdots  & \vdots  & \vdots & \vdots  & \vdots  & \vdots  & \ddots \\
\end{array}
$$

\begin{defi} The \emph{barycentric subdivision of} $a\in ({\bf Z}^{\mathbb{N}_+^{op}})_{\text{fin}}$ is the augmented sequence defined as
$\text{sd} (a):=a \tilde{\triangleright}\text{cad}^+.$ This construction gives rise to a functor $$\text{sd} \colon ({\bf Z}^{\mathbb{N}_+^{op}})_{\text{fin}} \to ({\bf Z}^{\mathbb{N}_+^{op}})_{\text{fin}}$$
\end{defi}

As $\text{bin}^{-1} \cdot  \text{cad}^+=\breve{\text{cad}}^+$, we also denote $\breve{\text{sd}}=\breve{\text{cad}}^+$ and we have an alternative description of $\text{sd} (a)$:
\begin{cor}  The barycentric subdivision of an augmented integer sequence $a\in  ({\bf Z}^{\mathbb{N}_+^{op}})_{\text{fin}}$ can be described by
$\text{sd}(a)= a \cdot \breve{\text{cad}}^+ = a \cdot \breve{\text{sd}}.$
\end{cor}

\begin{obs}

Given any augmented sequence $c=(c_{-1}, c_0, c_1, \cdots, c_n, 0,\cdots )\in  ({\bf Z}^{\mathbb{N}_+^{op}})_{\text{fin}}$ we have the formula
$$(\text{sd}c)_j=\sum_{i=-1}^{n}(j+1)! c_{i} \Stirling2{i+1}{j+1}.$$
In particular, the first three terms of this sequence are:
$$\begin{array}{l} (\text{sd}c)_{-1}=\sum_{i=-1}^{n}0!c_{i} \Stirling2{i+1} {0})= c_{-1}+0+\cdots 0=c_{-1} \\
(\text{sd}c)_0=\sum_{i=-1}^{n}1!c_{i} \Stirling2{i+1}{ 1}=c_0+\cdots+c_n \\
(\text{sd}c)_1=\sum_{i=-1}^{n}2!c_{i} \Stirling2{i+1}{ 2}= 2 c_{1}+ 2 c_2 \Stirling2{3}{2}+\cdots+2 c_n \Stirling2{n+1}{2}.
\end{array}$$


\end{obs}

\subsection{Comparing geometric and arithmetic constructions through the sequential cardinal functor}

From propositions \ref{cilCil}, \ref{cilCilcero} \ref{cilCildos}, we already know that $\text{cil}= |\text{Cil}|$, $\text{cil}_0= |\text{Cil}_0|$ and $\text{cil}_2= |\text{Cil}_2|$.
The following theorem gives an interesting relationship between geometric and arithmetic cylinders:

\begin{teo} If $X \in (\mathbf{Sets}_{\text{fin}})^{\mathbf{\Gamma}_+^{op}}$, then

\begin{itemize}
  \item[(i)] $|\text{Cil} (X)|=|X \tilde{ \vartriangleright} \text{Cil}|=|X| \tilde{\triangleright} |\text{Cil}|$,
    \item[(ii)]   $|\text{Cil}_0 (X)|=|X \tilde{ \vartriangleright} \text{Cil}_0|=|X| \tilde{\triangleright} |\text{Cil}_0|$,
  \item[(iii)] $|\text{Cil}_2 (X)|=|X \tilde{ \vartriangleright} \text{Cil}_2|=|X| \tilde{\triangleright} |\text{Cil}_2|$.
\end{itemize}
Therefore, the following diagrams are commutative:
  $$
\xymatrix{\ar[d]^{|\cdot|} (\mathbf{Sets}_{\text{fin}})^{\mathbf{\Gamma}_+^{op}} \ar[r]^{\text{Cil}} & (\mathbf{Sets}_{\text{fin}})^{\mathbf{\Gamma}_+^{op}}\ar[d]^{|\cdot|} \\
{\bf Z}^{{\mathbb{N}_+^{op}}}\ar[r]^{\text{cil}} & {\bf Z}^{{\mathbb{N}_+^{op}}}}\quad
\xymatrix{\ar[d]^{|\cdot|} (\mathbf{Sets}_{\text{fin}})^{\mathbf{\Gamma}_+^{op}} \ar[r]^{\text{Cil}_0} & (\mathbf{Sets}_{\text{fin}})^{\mathbf{\Gamma}_+^{op}}\ar[d]^{|\cdot|} \\
{\bf Z}^{{\mathbb{N}_+^{op}}}\ar[r]^{\text{cil}_0} & {\bf Z}^{{\mathbb{N}_+^{op}}}}$$
$$\xymatrix{\ar[d]^{|\cdot|} (\mathbf{Sets}_{\text{fin}})^{\mathbf{\Gamma}_+^{op}} \ar[r]^{\text{Cil}_2} & (\mathbf{Sets}_{\text{fin}})^{\mathbf{\Gamma}_+^{op}}\ar[d]^{|\cdot|} \\
{\bf Z}^{{\mathbb{N}_+^{op}}}\ar[r]^{\text{cil}_2} & {\bf Z}^{{\mathbb{N}_+^{op}}}.}
$$

\end{teo}

\begin{proof} (i), (ii) and (iii) follow from Theorem \ref{tildebigodottildedot}, the commutativity of the diagrams being a direct consequence.
\end{proof}

\begin{ej} Consider $H$ the augmented semi-simplicial hexaedron. Then (see Figure \ref{cerocilcildoscil}):
\begin{itemize}
\item  $|{{\rm{Cil}}} (H)|= {{\rm{Cil}}} (|H|)=|H|\cdot \breve{{{\rm{Cil}}}}$, and this matrix multiplication is
  $$|H|\cdot \breve{{{\rm{Cil}}}} =\\(1,6,6,0,0,0,\cdots ) \left(\begin{array}{ccccc}1 & 0 & 0 & 0 &\cdots \\0 & 2 & 1 &0 &\cdots \\0 & 0 & 3 & 2 &\cdots \\ \vdots &  \vdots  &  \vdots  &  \vdots  &\ddots \end{array}\right)= (1,12,24,12,0,0,0,\cdots )$$

\item  $|{{\rm{Cil}}_0} (H)|= {{\rm{Cil}}}_0 (|H|)=|H|\cdot \breve{{{\rm{Cil}}}}_0 $ and
  $$|H|\cdot \breve{{{\rm{Cil}}}}_0=\\(1,6,6,0,0,0,\cdots ) \left(\begin{array}{ccccc}1 & 0 & 0 & 0 &\cdots  \\0 & 2 & 0 &0 &\cdots   \\0 & 0 & 3 & 0 &\cdots \\ \vdots &  \vdots  &  \vdots  &  \vdots  &\ddots \end{array}\right)= (1,12,18,0,0,0,\cdots )$$
\item  $|{{\rm{Cil}}_2} (H)|= {{\rm{Cil}}}_2 (|H|)=|H|\cdot \breve{{{\rm{Cil}}}}_2$ and
$$|H|\cdot \breve{{{\rm{Cil}}}}_2  =\\(1,6,6,0,0,0,\cdots ) \left(\begin{array}{cccccc}1 & 0 & 0 & 0 &0 &\cdots \\0 & 2 & 1 &0 & 0 &\cdots   \\0 & 0 & 4 & 4 & 1 &\cdots \\ \vdots &  \vdots  &  \vdots  &  \vdots  & \vdots  &\ddots \end{array}\right)= (1,12,30,24,6,0,0,0,\cdots )$$

\end{itemize}
\end{ej}

\begin{figure}[htbp]
\begin{center}
\includegraphics[scale=0.7]{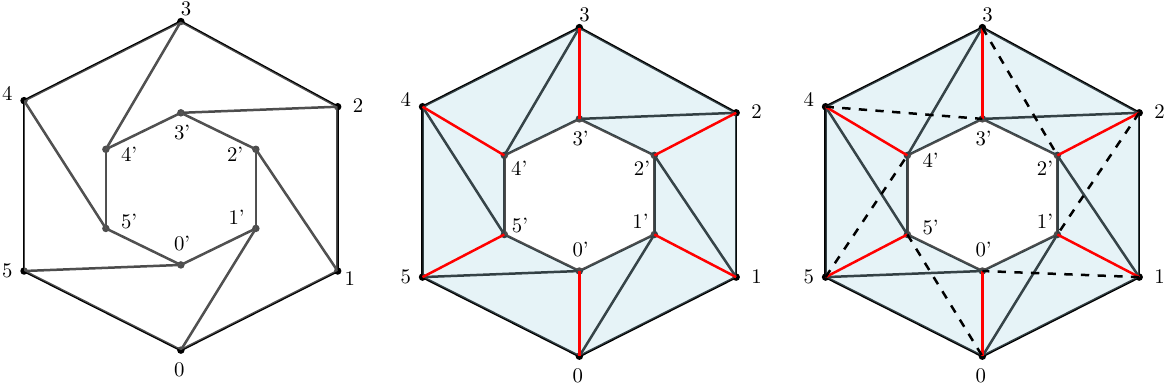}
\caption{From left to right ${{\rm{Cil}}}_0 (H)\subset {{\rm{Cil}}} (H) \subset  {{\rm{Cil}}}_2 (H)$}
\label{cerocilcildoscil}
\end{center}
\end{figure}


\bigskip
Now we compare geometric and arithmetic subdivisions through the sequential cardinal functor. We start with this simple lemma:

\begin{lem}\label{Sdcadplus} $|\text{Sd}{\Gamma_+}[n]|=\text{cad}^+[n]$, for all $n \in \mathbb{N}_+$.
\end{lem}

\begin{proof}
Associated to any $a \in \text{Sd}{\Gamma_+}[n]_m$, represented as
$$\emptyset =[-1]\stackrel{\varphi _{-1}}{\longrightarrow }[k_0]\stackrel{\varphi _0}{\longrightarrow }[k_1]\stackrel{\varphi _1}{\longrightarrow }\cdots
\stackrel{\varphi _{m-1}}{\longrightarrow }[k_m]\stackrel{\varphi _m}{\longrightarrow }[k_{m+1}]=[n],$$
\noindent we consider the subset chain
$${\emptyset=[-1] \subset \text{im}(\varphi_0) \subset \text{im}(\varphi_1) \subset \cdots \subset \text{im}(\varphi_m)\subset [n]}$$
This correspondence gives a bijection between $\text{Sd}{\Gamma_+}[n]_m$ and $\text{Cad}^+[n]_m$, for all $m$.
\end{proof}

\bigskip
Finally, we give a result giving a nice relationship between geometric and arithmetic subdivisions.
We use the notation $\text{sd}\gamma_+[n]:=|\text{Sd} ({\Gamma_+}[n])|$ and take $\text{sd}:=|\text{Sd}|$, the correspondent matrix.

\begin{teo} \label{cadtheorem}If $X \in (\mathbf{Sets}^{\mathbf{\Gamma}_+^{op}})_{\text{fin}}$ and $a \in ({\bf Z}^{{\mathbb{N}_+^{op}}})_{\text{fin}}$, then

\begin{itemize}
\item[(i)] $|\text{Sd} (X)|=|X \tilde{ \vartriangleright} \text{Sd}|=|X| \tilde{\triangleright} |\text{Sd}|$,
\item[(ii)] $\text{sd}=|\text{Sd}|=\text{cad}^+$,
\item[(iii)] $\text{sd}(a)= a \tilde{\triangleright} \text{sd}= a \cdot \breve{\text{sd}}$.
\end{itemize}
Moreover, the following diagram is commutative:
  $$
\xymatrix{\ar[d]^{|\cdot|} (\mathbf{Sets}^{\mathbf{\Gamma}_+^{op}})_{\text{fin}}\ar[r]^{\text{Sd}} &(\mathbf{Sets}^{\mathbf{\Gamma}_+^{op}})_{\text{fin}}\ar[d]^{|\cdot|} \\
({\bf Z}^{{\mathbb{N}_+^{op}}})_{\text{fin}}\ar[r]^{\text{sd}} & ({\bf Z}^{{\mathbb{N}_+^{op}}})_{\text{fin}}}
$$
\end{teo}

\begin{proof} (i) follows by Definition \ref{subd-semi} and Theorem \ref{tildebigodottildedot}.
By Lemma \ref{Sdcadplus}, we have $|\text{Sd}{\Gamma_+}[n]_m|=\text{cad}^+[n]_m$, and therefore $|\text{Sd}{\Gamma_+}[n]|=\text{cad}^+[n]$ and $|\text{Sd}|=\text{cad}^+$; so (ii) holds. Finally, (iii) and the commutativity of the diagram are consequences of (i) and (ii).

\end{proof}

\section{Conclusions and future work}
In this work, we have demonstrated how the subdivision and cylinder constructions for semi-simplicial sets can be obtained by taking certain actions of appropriate co-semi-simplicial objects. Additionally, we have discussed the computation of the sequential cardinality of cylinders and the sequential cardinality of barycentric subdivisions.

The sequential cardinality of cylinders can be easily computed using locally finite matrices, which are matrices with rows and columns that are eventually zero. Furthermore, the sequential cardinality of the barycentric subdivision can be computed using chain-power numbers and Stirling numbers (see \cite{Co74} and \cite{Schader1980}).

Another interesting objective would be to construct a categorical semi-ring (or symmetric bimonoidal category structure)  $$((\mathbf{Sets}_{\text{fin}})^{\mathbf{\Gamma}_+^{op}},  \boxplus, \bigodot,   \Gamma_+[-1],  \Gamma_+[0])$$ that satisfies the following properties:
\begin{itemize}
  \item[] If $X, Y\in (\mathbf{Sets}_{\text{fin}})^{\mathbf{\Gamma}_+^{op}}$ and $\text{dim}(X), \text{dim}(Y)$ are finite, then
  $$ \text{dim}(X \bigodot Y)=(\text{dim}(X)+1) (\text{dim}(Y)+1) -1.$$
  \item[ ]If $X, Y, Z \in (\mathbf{Sets}_{\text{fin}})^{\mathbf{\Gamma}_+^{op}}$, then
   $$(X \boxplus Y) \bigodot Z \cong (X\bigodot  Z) \boxplus  (Y\bigodot  Z). $$
\end{itemize}

If we are able to develop this construction, we will analyze the following:

(i) the relationship between $|X \bigodot Y|$ and $|X|$ and $|Y|$.

Another intriguing aim is to consider the Betti sequences of a finite augmented semi-simplicial set $\beta(X)=(\beta_{-1}(X), \beta_{0}(X), \beta_{1}(X), \cdots )$ and to study

(iii)  the relationship between $\beta(X \boxplus Y)$ and $\beta(X)$ and  $\beta(Y)$, and

(iii)  the relationship between $\beta(X \bigodot Y)$ and $\beta(X)$ and  $\beta(Y)$.

{\bf Funding:} {This research has been funded by the project PID2020-118753GB-I00 of the Spanish Ministry of Science and Innovation, the project REGI22-63 of the University of La Rioja, and the University of La Laguna.}

%

\end{document}